\documentclass[reqno, 10pt]{amsart}
\usepackage{amssymb,amsmath,amsfonts,bm, mathtools,dsfont}
\usepackage{dsfont}
\usepackage{epsfig}
\usepackage{graphicx}
\usepackage{esint}
\usepackage{enumerate}
\usepackage{verbatim}
\usepackage{color}
\usepackage{caption}
\usepackage{geometry}
\usepackage{enumitem}
\usepackage{lipsum}

\usepackage{hyperref,  circledsteps,}



\def\R{{\mathbb R}}
\def\S{{\mathbb S}}
\def\N{{\mathbb N}}

\def\t{{\underline t }}
\def\l{{\langle }}
\def\r{{\rangle }}

\def\I{{\mathcal I}}
\def\C{{\mathfrak{C}}}
\def\1{{\mathds{1}}}

\date{} 

  \DeclareMathOperator*{\sign}{sgn}

\DeclareMathOperator\support{supp}

\theoremstyle{plain}
\newtheorem{theorem}{Theorem}[section]
\newtheorem{definition}[theorem]{Definition}
\newtheorem{proposition}[theorem]{Proposition}

\newtheorem{lemma}[theorem]{Lemma}

\newtheorem{remark}[theorem]{Remark}
 \newcommand{\tild}{\widetilde}

\numberwithin{theorem}{section}
\numberwithin{equation}{section}
\numberwithin{figure}{section}

\let\oldtocsection=\tocsection
 
\let\oldtocsubsection=\tocsubsection
 
\let\oldtocsubsubsection=\tocsubsubsection
 
\renewcommand{\tocsection}[2]{\hspace{0em}\oldtocsection{#1}{#2}}
\renewcommand{\tocsubsection}[2]{\hspace{1em}\oldtocsubsection{#1}{#2}}
\renewcommand{\tocsubsubsection}[2]{\hspace{2em}\oldtocsubsubsection{#1}{#2}}

\begin{document}
\bibliographystyle{plain}

\parskip=4pt

\vspace*{1cm}
\title[Global  existence and uniqueness of solutions to the wave kinetic hierarchy]
{Inhomogeneous wave kinetic equation and its hierarchy in polynomially weighted $L^\infty$ spaces}

\author[Ioakeim Ampatzoglou]{Ioakeim Ampatzoglou}
\address{Ioakeim Ampatzoglou,  
Baruch College, The City University of New York}
\email{ioakeim.ampatzoglou@baruch.cuny.edu}

\author[Joseph K. Miller ]{Joseph K. Miller}
\address{Joseph K. Miller,  
Department of Mathematics, The University of Texas at Austin.}
\email{jkmiller@utexas.edu}

\author[Nata\v{s}a Pavlovi\'{c}]{Nata\v{s}a Pavlovi\'{c}}
\address{Nata\v{s}a Pavlovi\'{c},  
Department of Mathematics, The University of Texas at Austin.}
\email{natasa@math.utexas.edu}

\author[Maja Taskovi\'{c}]{Maja Taskovi\'{c}}
\address{Maja Taskovi\'{c},  
Department of Mathematics, Emory University}
\email{maja.taskovic@emory.edu}
\begin{abstract}

Inspired by ideas stemming from the analysis of the Boltzmann equation, in this paper we expand
well-posedness theory of the spatially inhomogeneous 4-wave kinetic equation, and also analyze an infinite hierarchy of PDE associated with this nonlinear equation.
More precisely, we show global in time well-posedness of the spatially inhomogeneous  4-wave kinetic equation    for polynomially decaying initial data. 
For  the associated infinite hierarchy,   we construct global in time solutions using the solutions of the wave kinetic equation and the Hewitt-Savage theorem. Uniqueness of these solutions  is proved by using a combinatorial board game argument tailored to this context, which allows us to control the factorial growth of the Dyson series.

\end{abstract}
\maketitle
\tableofcontents

\section{Introduction}

The goal of this paper is to establish global well-posedness results for certain partial differential equations that appear in the context of wave turbulence, so we start by briefly reviewing some mathematical results pertaining to wave turbulence.

\subsection{A framework of wave turbulence}
When faced with a dynamical system containing a large number of nonlinear interacting waves, instead of considering individual 
trajectories it is often helpful to study ensembles. Such a statistical mechanic treatment of the dynamics is what is usually referred to as wave turbulence, the main topic
of which is focused on rigorously deriving and analyzing an effective equation for the non-equilibrium dynamics of the relevant microscopic system.
While appearance of a wave kinetic equation goes back to works of Peierls in 1920 \cite{pe29} and Hasselmann in early 1960s \cite{ha62, ha63}, revived interest was 
inspired, in part, by the influential work \cite{zalvfa92} of Zakharov, L'vov and Falkovich in 1992 that uncovered 
a power-law type stationary solutions analogous to the
Kolmogorov spectra of hydrodynamic turbulence. For more details,  see for example Nazarenko \cite{na11} and Newell-Rumpf \cite{neru13}.

The attention of the mathematical community in wave turbulence has largely focused, so far, on the derivation of wave  kinetic equations starting from dynamics governed by nonlinear dispersive equations. In a pioneering work  \cite{lusp11}, Lukkarinen and Spohn  obtained a  rigorous derivation of  a linearized spatially homogeneous wave kinetic equation from a cubic nonlinear Scr\"{o}dinger equation (NLS) at statistical equilibrium, by employing Feynmann diagrams expansions. For 3-wave systems, linearized wave turbulence convergence results were also obtained by Faou \cite{fa20}. Regarding the out of the equilibrium case, starting from a cubic NLS with random data out of equilibrium, emergence of the 4-wave  spatially homogeneous kinetic equation
  \begin{align} \label{intro - h4KWE}
    \partial_t f = \mathcal{C}[f],
\end{align}
with  the collision operator defined as in \eqref{KWE collision operator},  has been studied in a sequence of papers \cite{bugehash21, coge19, coge20, deha21} that led to the works of Deng-Hani \cite{deha23, deha22}, who obtained a derivation up to the kinetic time\footnote{ The kinetic time is the relevant time scale for which one expects the system to exhibit a kinetic behavior.} and their recent work \cite{deha23long} where the derivation is obtained as long as the nonlinear wave kinetic equation \eqref{intro - h4KWE} is well-posed.
We would also like to note that starting from a stochastic Zakharov-Kuznetsov equation (which is a multidimensional generalization of Korteveg-de-Vries equation) with multiplicative noise, Staffilani-Tran \cite{sttr21} derived the spatially homogeneous 3-wave  kinetic equation up to the kinetic time.

The spatially {\it inhomogeneous} wave kinetic equation 
\begin{align} \label{intro - general KWE}
    \partial_t f+v\cdot\nabla_x f= \mathcal{C}[f],
\end{align}
with operator $\mathcal{C}$ describing interaction of waves, appears in the physics literature \cite{zalv75, zalvfa92}.
This type of equation is used for modeling ocean waves \cite{smja13}. Moreover, Spohn \cite{sp06} discusses the emergence of an inhomogeneous phonon Boltzmann equation  and addresses its connection to nonlinear waves. 
The first rigorous derivation result regarding the spatially  inhomogeneous wave kinetic equation was obtained by the first author of this paper, Collot and Germain  \cite{amcoge24}, who derived a   3-wave kinetic equation from quadratic  Scr\"{o}dinger-type nonlinearities. On the other hand, starting from a stochastic Zakharov-Kuznetsov equation with multiplicative noise,   Hannani-Rosenzweig-Staffilani-Tran \cite{harosttr22} derived a spatially inhomogeneous 3-wave kinetic equation up to the kinetic time. Recently, Hani-Shatah-Zhu \cite{hashzh23} derived several inhomogeneous and homogeneous wave kinetic equations from the simplified model of the Wick NLS whose main feature is the absence of all self-interactions in the correlation expansions of its solutions.

  Despite exciting activity on the rigorous derivation of wave kinetic equations, to the best of our knowledge, the analysis of wave kinetic equations themselves  has been carried out only in some instances. For example,  we note that  Escobedo-Vel\'azquez \cite{esve15}  constructed solutions to the spatially homogeneous bosonic Nordheim equation that exhibit blow up in finite time. 
 Germain-Ionescu-Tran \cite{geiotr20}  proved local well-posedness of the spatially homogeneous 4-wave kinetic equations in $L_v^2$ and $L_v^\infty$ based spaces. Moreover, Menegaki \cite{me23} showed $L^2$-stability near equilibrium,  and Collot-Dietert-Germain  \cite{codige24} showed stability and cascades of the Kolomogorov-Zakharov spectra, all in the cases of spatially homogeneous equations.
    On the other hand, for the spatially inhomogeneous equation, we are aware only of the work  the first author of this paper \cite{am24}, who recently obtained global well-posedness for the  4-wave kinetic equation in exponentially weighted $L_{x,v}^\infty$ spaces, employing classical tools of the kinetic theory of particles.

The aim of this paper is  to broaden the analysis of the inhomogeneous 4-wave kinetic equation and  its associated infinite hierarchy of PDE - the inhomogeneous 4-wave kinetic hierarchy.
We start by presenting an overview of the results of this paper. Notation and precise statements of results are then presented in subsections \ref{intro-subsection-KWE} and \ref{intro-subsection-KWH}.

\subsection{Results of this paper in a nutshell.}  \label{nutshell}
\begin{enumerate}
\item 
Inspired by ideas stemming from the analysis of the Boltzmann equation, in this paper we expand  well-posedness theory of the spatially inhomogeneous 4-wave kinetic equation
  \begin{align} \label{intro - 4KWE}
    \partial_t f+v\cdot\nabla_x f= \mathcal{C}[f],
\end{align}
where the collision operator is defined in \eqref{KWE collision operator}.
  In particular, we  prove existence of a unique global in time solution to the corresponding integral equation in {\it polynomially weighted} $L_{x,v}^\infty$ spaces (see Definition \ref{def-kwe mild solution equation} for the precise definition of the solution and Theorem \ref{thm-kwe is well posed} for the well-posedness result for the equation).  This result is a consequence of a novel a priori bound (Proposition \ref{a-priori estimate equation}), which is motivated by the a priori bound obtained by Toscani \cite{to86} for solutions of the Boltzmann equation.
  However, in the current context of the wave kinetic equation \eqref{intro - 4KWE}, by exploiting higher order multilinear operators   \eqref{gain loss form}, we were able to rely on the integrals with higher order velocity weights (see Lemma \ref{appendix lemma on velocities weight}). Thanks to these additional weights, we did not need Carleman-like representation, as was the case for the Boltzmann equation in \cite{to86, ammipata24}

\item 
Furthermore, in this paper we also study the spatially inhomogeneous 4-wave kinetic hierarchy
\begin{align} \label{intro - KWH - first}
   \partial_t f^{(k)} + \sum_{j=1}^k v_k\cdot \nabla_{x_k} f^{(k)} = \C^{(k+2)}f^{(k+2)},
\end{align}
with $\C^{(k+2)}$ given by \eqref{C^k}. 
To the best of our knowledge, this is the first paper that analyzes a spatially inhomogeneous wave kinetic hierarchy. 
We note that homogeneous version of the 4-wave kinetic hierarchy has been studied by Rosenzweig-Staffilani in \cite{rost22}, who obtained a local in time existence and uniqueness of solutions in polynominally weighted $L^{\infty}_v$ spaces. Also, Deng-Hani \cite{deha22} derived{\footnote{This kind of derivation is in contrast to derivations of infinite Boltzmann and Schr\"{o}dinger hierarchies from particles systems, where corresponding nonlinear equations are obtained from infinite hierarchies. Instead, in \cite{deha22}, authors utilize the derivation of nonlinear wave kinetic equation in order to derive associated infinite hierarchy.}} spatially homogeneous 4-wave kinetic {\it hierarchy} based on a derivation of 4-wave kinetic equation and a propagation of chaos for this equation.

The main objects of this paper - the equation \eqref{intro - 4KWE} and the hierarchy \eqref{intro - KWH - first} - are connected by the fact that the hierarchy \eqref{intro - KWH - first} admits a special class of factorized solutions
\begin{align}
    f^{(k)}(t, X_k, V_k) = \prod_{j=1}^k f(t,x_j,v_j),
\end{align}
 with each factor $f$ solving the 4-wave kinetic equation \eqref{intro - 4KWE},  where $X_k=(x_1, \dots, x_k)$ and $V_k=(v_1, \dots, v_k)$. We note that this factorization is analogous to the relationship between solutions of the Boltzmann equation and the Boltzmann hierarchy\footnote{which is an infinite hierarchy of coupled linear equations appearing in a rigorous derivation of the Boltzmann equation from many particle systems} \cite{la75, ki75, gastte13, pusasi14, chho23}  as well as the relationship between solutions of the NLS and the infinite hierarchy that appears in the derivation of the NLS from quantum many particle systems, referred to as  the Gross-Pitaevskii hierarchy \cite{erscya06, erscya07, chpa11, chpa14, so15, kiscst11, amliro20}. In contrast to known derivations of the Boltzmann and Gross-Pitaevskii hierarchies from many particle systems, the derivation of the inhomogeneous  wave kinetic hierarchy \eqref{intro - KWH - first} is an open problem.

We now summarize main results of this paper pertaining to the spatially inhomogeneous wave kinetic hierarchy \eqref{intro - KWH - first}.
\begin{enumerate}
\item 
Since in this paper we first establish the existence of global in time solutions to the wave kinetic equation \eqref{intro - 4KWE}, we can use these solutions to construct a global in time solution of the wave kinetic hierarchy \eqref{intro - KWH - first} for a special class  of initial data, the so-called admissible\footnote{In instances of particle systems, admissible data can be thought of as the marginals of a probability density.} initial data (for the precise statement, see Definition \ref{admissible data}). 
This construction is implemented in a similar fashion as in our recent work \cite{ammipata24} on the well-posedness for the Boltzmann hierarchy. More precisely, 
\begin{enumerate}
    \item[Step 1:] Since initial data  $F_0=(f_0^{(k)})_{k=1}^\infty$  of the wave kinetic hierarchy \eqref{intro - KWH - first} is assumed to be admissible, by the Hewitt-Savage theorem \cite{hs55}, there is a unique Borel probability measure  $\pi$ such that $F_0$ can be  represented  as a convex combination of tensorized states with respect to $\pi$  over the set of probability densities $\mathcal{P}$ i.e.
    \begin{equation*}
     f_0^{(k)} = \int_{\mathcal{P}} h_0^{\otimes k} d\pi(h_0).
    \end{equation*}
    It can be shown that the measure $\pi$ is  supported on a set of probability densities of a space-velocity polynomial decay (see Proposition \ref{hewitt-savage} below, which was proved in \cite{ammipata24}). 
\item[Step 2:] For each  $h_0$ in the support of the measure $\pi$, by the well-posedness result of this paper for the wave kinetic equation  (see Theorem \ref{thm-kwe is well posed}), there exists a global in time solution $h(t)$ to the  wave kinetic equation \eqref{intro - 4KWE}  of the same polynomial decay as the initial data. Finally, equipped with these solutions of the wave kinetic equation, we construct a solution $F=(f^{(k)})_{k=1}^\infty$ of the wave kinetic hierarchy \eqref{intro - KWH - first}  as follows:
\begin{equation}\label{intro-step2}
        f^{(k)}(t) : = \int_{\mathcal{P}} h(t)^{\otimes k} d\pi(h_0),\quad  k\in\N,
    \end{equation}  
and prove that it belongs to the class of polynomially weighted $L^\infty_{x,v}$ solutions of \eqref{intro - KWH - first}.
\end{enumerate}

\begin{remark} This two-step proof of existence of solutions to the inhomogeneous wave kinetic hierarchy \eqref{intro - KWH - first} is different than the  proof of existence of solutions to the homogeneous wave kinetic hierarchy by Rosenzweig-Staffilani in \cite{rost22}. Our approach utilizes the global in time existence of solutions to the corresponding nonlinear equation. Indeed, as it can be seen in \eqref{intro-step2}, our solution to the hierarchy is built from corresponding solutions to the nonlinear equation \eqref{intro - 4KWE}. On the other hand, the work \cite{rost22} employs representing a solution of the infinite hierarchy by iterating Duhamel formulas. In our case, such an iteration is not needed for proving {\it existence} of solutions for the infinite hierarchy \eqref{intro - KWH - first}.

\end{remark}

\item 
We also prove uniqueness of solutions to the wave kinetic hierarchy \eqref{intro - KWH - first}. While existence of solutions to the wave kinetic hierarchy is established  in this paper for admissible initial data, our uniqueness proof does not use admissibility. More precisely, we prove uniqueness of a mild solution (see Definition \ref{def-kwh mild solution}) to the wave kinetic hierarchy  \eqref{intro - KWH - first} corresponding to initial data in  polynomially weighted $L^\infty_{x,v}$ spaces. The main ingredients of the uniqueness proof are:
\begin{itemize}
\item[(i)] an a priori estimate (see Proposition \ref{a priori estimate}) for the wave kinetic hierarchy \eqref{intro - KWH - first}. 
We prove this estimate  by adapting to the context of the infinite hierarchy ideas  used to obtain the a priori bound (see Proposition \ref{a-priori estimate equation}) for the wave kinetic equation \eqref{intro - 4KWE};
\item[(ii)] a new combinatorial board game argument, inspired by: 
\begin{itemize} 
\item[$\bullet$] the board game argument introduced by Klainerman and Machedon  \cite{klma08} in the context of the Gross-Pitaevskii hierarchy corresponding to the cubic nonlinear Schr\"odinger equation, 
\item[$\bullet$] the adaptation of T. Chen and the third author of this paper  \cite{chpa11} of the board game for the Gross-Pitaevskii hierarchy corresponding to the quntic NLS, and 
\item[$\bullet$] our recent use \cite{ammipata24} of board game arguments in the context of the Boltzmann hierarchy in $L^{\infty}_{x,v}$-based spaces\footnote{We note that prior to \cite{ammipata24} all implementations of the board game argument were done in $L^2$-based spaces.}.
\end{itemize} 
At the heart of  board game arguments is a reorganization of the iterated Duhamel formulas (which contain a factorial number of terms) into an exponential number of equivalence classes (see Proposition \ref{prop-equivalent classes} which achieves that\footnote{A special upper echelon form  in this proposition can be understood as a representative of an equivalence class mentioned above.} for hierarchy \eqref{intro - KWH - first}).
We note that this paper presents the first application of a board game argument in the context of a wave kinetic hierarchy. 
\end{itemize}

\end{enumerate} 
\end{enumerate}

We continue the introduction by precisely describing our results for the wave kinetic equation in Section \ref{intro-subsection-KWE} and results for the wave  kinetic hierarchy in Section \ref{intro-subsection-KWH}.

\subsection{Wave kinetic equation: notation and the main result}\label{intro-subsection-KWE}
The Cauchy problem for the spatially inhomogeneous 4-wave kinetic equation  for a function  $f:[0,\infty)\times\R^{3}\times\R^{3}\to \R$ with initial data $f_0:\R^{3}\times\R^{3}\to \R$, is given by
\begin{equation}\label{KWE}
\begin{cases}
\partial_t f+v\cdot\nabla_x f= \mathcal{C}[f],  \\
f(t=0)=f_0,
\end{cases}
\end{equation}
with the collisional operator defined as follows
\begin{equation}\label{KWE collision operator}
\mathcal{C}[f]=\int_{\R^9}\delta(\Sigma)\delta(\Omega)ff_1f_2f_3\left(\frac{1}{f}+\frac{1}{f_1}-\frac{1}{f_2}-\frac{1}{f_3}\right)\,dv_1\,dv_2\,dv_3,    
\end{equation}
and where the resonant manifolds are given by
\begin{align}\label{manifold for equation}
    \Sigma=v+v_1-v_2-v_3,\quad \Omega=|v|^2+|v_1|^2-|v_2|^2-|v_3|^2.
\end{align}
We use the notation
$f:=f(v)$, $f_i:=f(v_i)$, $ i=1,2,3.$

The collisional operator $\mathcal{C}[f]$ can be equivalently written as follows:
\begin{equation}\label{gain loss form}
\mathcal{C}[f]=L_0(f,f,f)+L_1(f,f,f)-L_2(f,f,f)-L_3(f,f,f),    
\end{equation}
where for functions $g, h, l:[0,\infty)\times\R^{3}\times\R^{3}\to \R$ we denote
\begin{align}
L_0(g,h,l)&=\int_{\R^9}\delta(\Sigma)\delta(\Omega)g(v_1)h(v_2) l(v_3)\,dv_1\,dv_2\,dv_3, \label{L_0}\\
L_1(g,h,l)&=\int_{\R^9}\delta(\Sigma)\delta(\Omega)g(v) h(v_2) l(v_3)\,dv_1\,dv_2\,dv_3, \label{L_1}\\
L_2(g,h,l)&=\int_{\R^9}\delta(\Sigma)\delta(\Omega)g(v) h(v_1) l(v_3)\,dv_1\,dv_2\,dv_3, \label{L_2}\\
L_3(g,h,l)&=\int_{\R^9}\delta(\Sigma)\delta(\Omega)g(v) h(v_1) l(v_2)\,dv_1\,dv_2\,dv_3. \label{L_3}
\end{align}
 Notice that the operators $L_0,L_1,L_2,L_3$ are multilinear with respect to their arguments and monotone when the inputs are non-negative. 

 The collisional operator $\mathcal{C}[f]$ can  also be written in weak formulation as follows \cite[pp.122, eqn (8.18)]{na11}
\begin{equation}\label{weak form}
\int_{\R^3}\mathcal{C}[f]\phi\,dv= \int_{\R^{12}}\delta(\Sigma)\delta(\Omega)ff_1f_2f_3\left(\frac{1}{f}+\frac{1}{f_1}-\frac{1}{f_2}-\frac{1}{f_3}\right)\left(\phi+\phi_1-\phi_2-\phi_3\right)\,dv_1\,dv_2\,dv_3\,dv,
\end{equation}
where $\phi$ is a test function appropriate for all the above integrations to make sense. Choosing $\phi\in\{1,v,|v|^2\}$ and using the resonant conditions, one can formally see that a solution $f$ to \eqref{KWE}  conserves mass, momentum and energy:
\begin{equation}\label{conservation laws formal}\partial_t \int_{\R^3}f\phi\,dv=0,\quad \phi\in\{1,v,|v|^2\}.
\end{equation}

We now introduce the spaces that will be used in the formulation of the well-posedness result for the wave kinetic equation \eqref{KWE}. We shall be using the notation that for $y \in \R^3$, 
\begin{equation*}
    \langle y \rangle^2 = 1 + |y|^2
\end{equation*}
to define the following function space.

\begin{definition}
    For each $T > 0$, $p,q > 1$ and $\alpha,\beta > 0$ we define the space 
    \begin{equation*}
        X_{p,q, \alpha,\beta} : = \left\{ f: \R^3 \times \R^3 \rightarrow \R  \text{ measurable functions such that } \Vert f \Vert_{p,q,\alpha, \beta} < \infty \right\},
    \end{equation*}
    where we define the norm as 
    \begin{equation*}
        \Vert f \Vert_{p,q,\alpha,\beta} : = \left \Vert \langle \alpha x\rangle^p \langle \beta v \rangle^q f(x,v) \right \Vert_{L^\infty}.
    \end{equation*}
    We additionally define the functional spaces in time as 
    \begin{equation*}
        X_{p,q, \alpha,\beta, T} = \mathcal{C}([0,T], X_{p,q, \alpha, \beta}), 
    \end{equation*}
    where we are using the natural supremum norm on this space:
    \begin{equation*}
        ||| f(\cdot) |||_{p,q,\alpha,\beta,T} : = \sup_{t \in [0,T]} \Vert f(t) \Vert_{p,q, \alpha, \beta} .
    \end{equation*}
\end{definition}

Before we define a concept of the solution we will be working with, we note that we will denote by $T_1$ the  transport operator, which is defined by its action on a  function $g:[0,\infty)\times\R^{d}\times\R^{d}\to\R$ as follows:
\begin{equation}\label{definition of transport for KWE}
 T_1^{s}g^{(k)}(t,x,v):=g^{(k)}(t,x-sv,v).
\end{equation}
With this notation, by Duhamel's formula, the wave kinetic equation \eqref{KWE}  can be formally written in mild form as:
\begin{equation}\label{KWE - mild form classic}
f(t)=T_1^{t}f_0+\int_0^t T_1^{t-s}\mathcal{C}[f](s)\,ds,\quad t\in[0,T],
\end{equation}
or equivalently after applying $T_1^{-t}$,
\begin{equation}\label{KWE - mild form we use}
    T_1^{-t}f(t)=f_0 +\int_0^t T_1^{-s}\mathcal{C}[f](s)\,ds,\quad t\in[0,T].
\end{equation}
We can now state precisely the definition of a mild solution we will be using in this paper.

\begin{definition}[Mild solution of the wave kinetic  equation]
Let $T>0$, $p,q>1$ and $\alpha,\beta>0$, and consider initial data $f_0\in X_{p,q,\alpha,\beta}$. A measurable function $f:[0,T]\times\R^3\times\R^3\to\R$ is called a mild solution to the  wave kinetic  equation \eqref{KWE} in $[0,T]$ corresponding to the initial data $f_0$  if 
\begin{align}
   T_1^{-(\cdot)}f(\cdot)\in X_{p,q,\alpha,\beta,T},
\end{align} 
and 
\begin{equation}\label{def-kwe mild solution equation}
    T_1^{-t} f(t,x,v)= f_0+\int_0^t T_1^{-s} \mathcal{C}[f](s,x,v)\,ds,\quad t\in[0,T].
\end{equation}
\end{definition}

We now state the main result regarding the well-posedness of the wave kinetic equation \eqref{KWE}.

\begin{theorem} \label{thm-kwe is well posed}
 Let $p>1,q>3$, $\alpha,\beta>0$ and $T>0$. Let $M>0$ with $M<(24C_{p,q,\alpha, \beta})^{-1/2}$, where $C_{p,q,\alpha,\beta}$ is given by \eqref{constant C}. Consider $f_0\in X_{p,q,\alpha,\beta}$, with $\|f_0\|_{p,q,\alpha,\beta}\le \frac{M}{2}$. Then there exists a unique mild solution to the wave kinetic  equation \eqref{KWE}, in the class of functions satisfying:
\begin{equation}\label{condition on uniqueness equation}
   |||T_1^{-(\cdot)}f(\cdot)|||_{p,q,\alpha,\beta,T}\leq  M. 
\end{equation}

If $f_0\ge 0$, the solution remains non-negative. Additionally, assuming that $f$ and $g$ are the mild solutions corresponding to initial data $f_0$ and $g_0$ respectively, we have the continuity with respect to initial data estimate:
\begin{equation}\label{stability estimate}
   |||T_1^{-(\cdot)}f(\cdot)-T_1^{-(\cdot)}g(\cdot)|||_{p,q,\alpha,\beta,T}\leq 2 \|f_0-g_0\|_{p,q,\alpha,\beta}. 
\end{equation}
In particular 
\begin{equation} \label{bound wrt initial data equation}
    |||T_1^{-(\cdot)}f(\cdot)|||_{p,q,\alpha,\beta,T}\leq 2 \|f_0\|_{p,q,\alpha,\beta}.
\end{equation}
Moreover,  the solution satisfies the following conservation laws: for any $t\in[0,T]$ and a.e. $x\in\R^3$:
\begin{equation}\label{conservation of mass KWE}
\text{If $p>3,\,q>4$}:\,\,\int_{\R^3}f(t,x,v)\,dv=\int_{\R^3}f_0(x,v)\,dv,    
\end{equation}
\begin{equation}\label{conservation of momentum KWE}
\text{If $p>3,\,q>5$ }:\,\,\int_{\R^3}vf(t,x,v)\,dv=\int_{\R^3}vf_0(x,v)\,dv,   
\end{equation}
\begin{equation}\label{conservation of energy KWE}
\text{If $p>3,\,q>6$ }:\,\,\int_{\R^3}|v|^2f(t,x,v)\,dv=\int_{\R^3}|v|^2f_0(x,v)\,dv.    
\end{equation}
\end{theorem}

\begin{remark}
We note that we work in dimension $d=3$ because one of the key estimates (Lemma \ref{appendix lemma on velocities weight}) is established only for this dimension. 
\end{remark}

\subsection{Wave kinetic hierarchy: notation and the main result}\label{intro-subsection-KWH}
As mentioned in subsection \ref{nutshell},  we introduce the spatially inhomogeneous 4-wave kinetic hierarchy as follows:
\begin{align}\label{kwh}
\partial_t f^{(k)}(t,X_k, V_k) + \sum_{j=1}^k v_k\cdot \nabla_{x_k} f^{(k)} = \C^{k+2}f^{(k+2)},
\end{align}
with collision operator defined by
\begin{align}\label{C^k}
\C^{k+2}f^{(k+2)} &= \sum_{j=1}^k  \C_{j, k+2}f^{(k+2)}\\
 & := \sum_{j=1}^k \left( \C^{L_0}_{j, k+2}f^{(k+2)} +\C^{L_1}_{j, k+2} f^{(k+2)} - \C^{L_2}_{j, k+2} f^{(k+2)} -\C^{L_3}_{j, k+2}f^{(k+2)} \right), \label{sum of four}
\end{align}
where
\begin{align}
&C^{L_0}_{j, k+2} f^{(k+2)} (t,X_k, V_k) \nonumber \\
&= 
\int_{\R^{9}} dv_{k+1} dv_{k+2}dv_{k+3} \,\delta(\Sigma_{j, k+2})\, \delta(\Omega_{j, k+2})
\, f^{(k+2)}(t, X_k, x_j, x_j, \,\, V_k^{j,v_{k+1}},   v_{k+2}, v_{k+3}), \label{CM} \\
&\C^{L_1}_{j, k+2} f^{(k+2)}(t,X_k, V_k) \nonumber \\
& =\int_{\R^{9}} dv_{k+1} dv_{k+2}dv_{k+3} \,\delta(\Sigma_{j, k+2})\, \delta(\Omega_{j, k+2})
\, f^{(k+2)}(t, X_k, x_j, x_j, \,V_k,   v_{k+2}, v_{k+3}),
\label{CL1}\\
&\C^{L_2}_{j, k+2} f^{(k+2)}(t,X_k, V_k) \nonumber \\
& =\int_{\R^{9}} dv_{k+1} dv_{k+2}dv_{k+3} \,\delta(\Sigma_{j, k+2})\, \delta(\Omega_{j, k+2})
\, f^{(k+2)}(t, X_k, x_j, x_j, \,V_k,   v_{k+1}, v_{k+3}),
\label{CL2}\\
&\C^{L_3}_{j, k+2} f^{(k+2)}(t,X_k, V_k) \nonumber \\
& =\int_{\R^{9}} dv_{k+1} dv_{k+2}dv_{k+3} \,\delta(\Sigma_{j, k+2})\, \delta(\Omega_{j, k+2})
\, f^{(k+2)}(t, X_k, x_j, x_j, \,V_k,   v_{k+1}, v_{k+2}),
\label{CL3}
\end{align}
and
\begin{equation}\label{Sigma and Omega}
\begin{aligned}
  &\Sigma_{j, k+2}  = v_j + v_{k+1} - v_{k+2} - v_{k+3},\\
  &\Omega_{j, k+2} = |v_j|^2 + |v_{k+1}|^2 - |v_{k+2}|^2 - |v_{k+3}|^2,
\end{aligned}
\end{equation}
and
\begin{align}\label{V^j_k}
     V_k^{j,v_{k+1}} = (v_1, \dots, v_{j-1}, \,  \underbrace{v_{k+1}}_{j-\text{th}} , \, v_{j+1}, \dots , v_k).
\end{align}

One can also represent each operator $\C_{j,k+2}$ as a difference:
\begin{align}\label{gain-loss}
    \C_{j,k+2} = \C^+_{j,k+2}-\C^-_{j,k+2},
\end{align}
where
\begin{align}\label{gain}
    & \C^+_{j,k+2} = \C^{L_0}_{j,k+2}+\C^{L_1}_{j,k+2},\\
    & \C^-_{j,k+2} = \C^{L_2}_{j,k+2}+\C^{L_3}_{j,k+2}. \label{loss}
\end{align}
With this notation we have
\begin{align} \label{gain-loss sum}
\C^{k+2} = \sum_{j=1}^k \C^+_{j,k+2} - \C^-_{j,k+2}.
\end{align}

 Motivated by  \cite{ammipata24}, we define Banach spaces that will be used throughout the paper. For $Y_k = (y_1, y_2, \dots, y_k) \in \R^{3k}$, we define
 \begin{align}\label{double bracket}
     \l\l Y_k\r\r:=\prod_{i=1}^k\l y_i\r,\quad \l y_i\r:=\sqrt{1+|y_i|^2},\quad i=1,\cdots,k.
 \end{align}

We are ready to define the following Banach spaces and their corresponding norms:

\noindent $\bullet$ Given $k\in \N$, $p,q >1$, $\alpha,\beta>0$, we define
\begin{align}\label{X_k static}
 & X_{ p,q, \alpha, \beta}^k:=\left\{ g^{(k)}:\mathbb{R}^{3k}\times\R^{3k}\to\R,\ \text{measurable and symmetric : }
 \|g^{(k)}\|_{k, p,q, \alpha, \beta}<\infty\right\}, \\  
&     \|g^{(k)}\|_{k, p,q, \alpha, \beta}:=\sup_{X_k, V_{k}} \l\l \alpha X_k\r\r^p \l\l \beta V_k\r\r^q \left|g^{(k)}(X_k,V_k)\right|. \label{X_k static norm}
 \end{align}
 Here by symmetric, we mean
 \begin{equation}\label{symmetry assumption}
g^{(k)}\circ\sigma_k=g^{(k)}, 
~ \text{ for any permutation $\sigma_k$ of pairs of variables $\{x_i, v_i\}_{i=1}^k$}.  
 \end{equation}

\noindent $\bullet$ Given $T>0$, $k\in \N$, $p,q >1$, $\alpha,\beta>0$, we define
\begin{align}   &X_{p,q,\alpha,\beta,T}^k:=C([0,T],X_{p,q,\alpha,\beta}^k),\label{X_k time} \\
& |||g^{(k)}(\cdot)|||_{k,p,q,\alpha,\beta,T}:=\sup_{t\in[0,T]}\|g^{(k)}(t)\|_{k,p,q,\alpha,\beta}  \label{X_k time norm}\\
& \hspace{3.1cm}=\sup_{t\in[0,T]} \sup_{X_k, V_{k}} \l\l \alpha X_k\r\r^p \l\l \beta V_k\r\r^q \left|g^{(k)}(t,X_k,V_k)\right|.\nonumber
\end{align}

In order to be able to look at a solution  of the infinite hierarchy \eqref{kwh} as a single object $F = (f^{(k)})_{k=1}^\infty$, we also introduce the following spaces:

\noindent $\bullet$ Given $p,q >1$, $\alpha,\beta>0$, $\mu \in \R$, we define 
\begin{align}\label{X_infty static}
 &\mathcal{X}_{p,q,\alpha, \beta,\mu}^\infty:=\left\{G=(g^{(k)})_{k=1}^\infty\in\prod_{k=1}^\infty X_{p,q,\alpha,\beta}^k\,:\,\|G\|_{p,q,\alpha, \beta,\mu}<\infty\right\},   \\
&\|G\|_{p,q,\alpha,\beta, \mu}=\sup_{k\in\N}e^{\mu k}\|g^{(k)}\|_{k,p,q,\alpha,\beta}  \label{X_infty static norm} 
\\
& \hspace{1.8cm}= \sup_{k\in\N}e^{\mu k}\sup_{X_k, V_{k}} \l\l \alpha X_k\r\r^p \l\l \beta V_k\r\r^q \left|g^{(k)}(X_k,V_k)\right|.
\end{align}

\noindent $\bullet$ Given $T>0$, $p,q >1$, $\alpha,\beta>0$, we define
\begin{align}\label{X_infty time}
&\mathcal{X}_{p,q,\alpha,\beta,\mu,T}^\infty:=C([0,T],\mathcal{X}_{p,q,\alpha,\beta,\mu}^\infty)\\
&|||G(\cdot)|||_{p,q,\alpha,\beta,\mu,T}:=\sup_{t\in[0,T]}\|G(t)\|_{p,q,\alpha,\beta,\mu} \label{X_infty time norm}\\
&\hspace{2.8cm} = \sup_{t\in[0,T]} \sup_{k\in\N}e^{\mu k}
\sup_{X_k, V_{k}} \l\l \alpha X_k\r\r^p \l\l \beta V_k\r\r^q \left|g^{(k)}(t,X_k,V_k)\right|. \nonumber
\end{align}

With definition of functional spaces in hand,  we are ready to give a precise definition of mild solutions to the  wave kinetic hierarchy \eqref{kwh}, which we motivate by the following observation.
For $k\in \N$, we denote the transport operator  acting on a function $g^{(k)}:[0,\infty)\times\R^{3k}\times\R^{3k}\to\R$  by
\begin{equation}\label{definition of transport}
 T_k^{s}g^{(k)}(t,X_k,V_k):=g^{(k)}(t,X_k-sV_k,V_k).
\end{equation}
With this notation, in analogy to \eqref{KWE - mild form we use}, a mild solution of the wave kinetic hierarchy \eqref{kwh} is formally given by
\begin{align}\label{mild solution formula}
   T^{-t}_k f^{(k)}(t) =  f_0^{(k)} + \int_0^{t} T_k^{-s }\C^{k+2} f^{(k+2)}(s) ds, \quad \forall k \in \N, \,\, \forall\,t\in[0,T].
\end{align}

\begin{definition}[Mild solution to the  wave kinetic  hierarchy]\label{def-kwh mild solution}
Let $T>0$,  $p,q>1$, $\alpha, \beta>0$,  $\mu\in\R$, and consider initial data $F_0=(f_0^{(k)})_{k=1}^\infty\in \mathcal{X}_{p,q,\alpha,\beta,\mu}^\infty$. A sequence $F = (f^{(k)})_{k=1}^\infty$ of measurable functions $f^{(k)}:[0,T]\times\R^{3k}\times\R^{3k}\to\R$ is called a  mild $\mu$-solution to  the  wave kinetic  hierarchy \eqref{kwh}  in $[0,T]$,  corresponding to the initial data $F_0$, if 
\begin{equation}\label{transport in space}
\mathcal{T}^{-(\cdot)}F(\cdot):=(T_k^{-(\cdot)}f^{(k)}(\cdot))_{k=1}^\infty\in\mathcal{X}_{p,q,\alpha,\beta,\mu,T}^\infty,
\end{equation}
and
\begin{equation}\label{Boltzmann hierarchy k}
T_k^{-t}f^{(k)}(t)=f_0^{(k)}+\int_0^t T_k^{-s}\C^{k+2}f^{(k+2)}(s)\,ds,\quad\forall\, t\in[0,T],\quad\forall k\in\N.
\end{equation}

\end{definition}

Our main result for the wave kinetic hierarchy \eqref{KWE} concerns its well-posedness, which means  existence and stability of solutions and their uniqueness. While uniqueness part does not require special structure of the initial data, our existence part utilizes the concept of admissible initial data. In order to be able to state the entire well-posedness result, we first recall the notion of admissibility.
\begin{definition}[Admissibility]\label{admissible data}
The set of admissible functions, denoted by $\mathcal{A}$, is defined by
\begin{align}\label{def-admisibility}
    \mathcal{A} := \Big\{ (g^{(k)})_{k=1}^\infty \in \prod_{k=1}^\infty L^1_{X_k,V_k}:& ~ \forall k\in\N ~\text{we have} ~ g^{(k)}\geq 0,~ g^{(k)} ~ \text{is symmetric  \eqref{symmetry assumption}}, \nonumber \\
    & 
    \int_{\R^{6k}}g^{(k)}\,dX_k\,dV_k=1, ~
     g^{(k)} = \int_{\R^{6}} g^{(k+1)} dv_{k+1} dx_{k+1}\Big\}.
\end{align}
\end{definition}

We are now ready to state our  main result on the well-posedness of the  wave kinetic  hierarchy.
\begin{theorem}[Global well-posedness for the wave kinetic hierarchy]\label{thm-gwp for kwh}
 Consider the wave kinetic hierarchy \eqref{kwh} in dimension $d=3$.
Let $T>0$, $p>1$, $q>3$, $\alpha, \beta>0$  and let $\mu\in\R$ be such that
        $e^{2\mu}>32C_{p,q,\alpha,\beta},$
    where $C_{p,q,\alpha,\beta}$ is given by \eqref{constant C}.
Consider  admissible initial data $F_0=(f_0^{(k)})_{k=1}^\infty\in \mathcal{A}\cap \mathcal{X}_{p,q,\alpha,\beta,\mu'}^\infty$, where $\mu'=\mu+\ln 2$. Then, there exists a unique mild $\mu$-solution $F=(f^{(k)})_{k=1}^\infty$ of the wave kinetic hierarchy \eqref{kwh}, with $f^{(k)}\ge 0$ for all $k$. In addition, the solution  satisfies the estimate
\begin{equation}\label{stability estimate hierarchy}
|||\mathcal{T}^{-(\cdot)}F(\cdot)|||_{p,q,\alpha,\beta,\mu,T}\leq  1.
\end{equation}
Moreover, the following $k$-particle conservation laws hold for any $t\in[0,T]$ and a.e. $X_k\in\R^{3k}$:

\begin{align}
    \text{If $p>3,\,q>4$}:\quad      \int_{\R^{3k}} f^{(k)}(t,X_k,V_k)  \,dV_k &= 1,\label{conservation of mass: KWH}\\
\text{If $p>3,\,q>5$}:\quad     \int_{\R^{3k}} V_k f^{(k)}(t,X_k,V_k) \,dV_k &= \int_{\R^{3k}} V_k f^{(k)}_0(X_k,V_k)  \,dV_k, \label{conservation of momentum: KWH}\\
\text{If $p>3, q>6$}:\quad     \int_{\R^{3k}} |V_k|^2f^{(k)}(t,X_k,V_k) \,dV_k &=\int_{\R^{3k}} |V_k|^2f_0^{(k)}(X_k,V_k) \,dV_k. \label{conservation of energy: KWH}
\end{align}

In the case that the initial data are tensorized i.e. $F_0=(f_0^{\otimes k})_{k=1}^\infty\in \mathcal{A} \cap \mathcal{X}^\infty_{p,q,\alpha,\beta,\mu'}$, there holds  the  stability estimate
\begin{equation}\label{stability estimate hierarchy KWH}
|||\mathcal{T}^{-(\cdot)}F(\cdot)|||_{p,q,\alpha,\beta,\mu,T}\leq \|F_0\|_{p,q,\alpha,\beta,\mu'} .  
\end{equation}
\end{theorem}

\subsection*{Organization of the paper}
In Section \ref{sec-WKE} we address global well-posedness of the wave kinetic equation. Global well-posedness of the wave kinetic hierarchy is proved in  Section \ref{sec-KWH}. In the appendix we gather auxiliary results used throughout the paper, including properties of tensorized functions in Appendix \ref{sec-tensorized}, properties of the resonant manifold in Appendix \ref{sec-resonant manifold}, and various integral estimates in Appendix \ref{sec-integral estimates}.

\subsection*{Acknowledgements}
N.P. and M.T. thank Luisa Velasco for helful discussions on  wave kinetic equations.
I.A. gratefully acknowledges support from the NSF grants No. DMS-2418020, DMS-2206618. J.K.M. gratefully acknowledges
support from the Provost’s Graduate Excellence Fellowship at The University of Texas at Austin
and from the NSF grants No. DMS-1840314, DMS-2009549 and DMS-2052789.
N.P. gratefully acknowledges
support from the NSF under grants No. DMS-1840314, DMS-2009549 and DMS-2052789. 
M.T. gratefully acknowledges support from the NSF grant DMS-2206187.

\section{Global well-posedness of the  wave kinetic equation}\label{sec-WKE}
 In this section, we obtain global in time existence, uniqueness and stability of mild solutions of the inhomogeneous wave kinetic equation \eqref{KWE} for initial data polynomially close to vacuum.  When the initial data are non-negative, we show that the corresponding solution remains non-negative, which is physically anticipated since the equation describes a localized point energy spectrum. Additionally, we prove that the solution conserves mass, momentum and energy for sufficiently decaying initial data.

The strategy for proving the global well-posedness of \eqref{KWE} relies on the following global in  time a-priori estimate.
\begin{proposition}\label{a-priori estimate equation}
 Let $p>1, q>3$, $\alpha,\beta>0$ and $T>0$. For any $j=0,1,2,3$, and $t\in[0,T]$,  the following bound holds
\begin{equation}\label{estimate for L eq}
\left\|\int_0^t T_1^{-s}L_j(g,h,l)(s)\,ds\right\|_{p,q,\alpha,\beta}  \le C_{p,q,\alpha,\beta} |||T_1^{-(\cdot)}g|||_{p,q,\alpha,\beta,T}|||T_1^{-(\cdot)}h|||_{p,q,\alpha,\beta,T}|||T_1^{-(\cdot)}l|||_{p,q,\alpha,\beta,T} 
\end{equation}
where $L_j$, with $j\in \{0,1,2,3\}$ are defined in \eqref{L_0}-\eqref{L_3}.
One can take
\begin{align}\label{constant C}
    C_{p,q,\alpha,\beta} = \frac{16 \, p \, \pi^3}{\alpha(p-1)} ~  \left(\frac{1}{3} + \frac{1}{q-3}\right)  ~\max\{\beta^q, \beta^{-3q}\}.
\end{align}
\end{proposition}

\begin{remark}
     The proof of this Proposition is inspired by the work of Toscani \cite{to86} for the Boltzmann equation. However, the presence of higher order multilinear operators $L_j$, $j \in \{0,1,2,3\}$ in the wave kinetic equation allows simplifications, as mentioned in subsection \ref{nutshell}.
\end{remark}

\begin{proof}

Without loss of generality, assume that
\begin{equation}\label{normalization}
|||T_1^{-(\cdot)}g|||_{p,q,\alpha,\beta,T}=|||T_1^{-(\cdot)}h|||_{p,q,\alpha,\beta,T}=|||T_1^{-(\cdot)}l|||_{p,q,\alpha,\beta,T}=1.
\end{equation}
 Given $v,v_1,v_2,v_3\in\R^3$, we will use the following notation
$$u_{i,j}:=v_i-v_j,\quad i,j\in\{0,1,2,3\},\,\, i\neq j,$$
where we abbreviate notation denoting $v_0=v$.

\noindent $\bullet$ {\em Proof for $L_0$.}
For any $x,v\in\R^{3}$ and $t\in[0,T]$, we have
\begin{align}
& I_{L_0}(t,x, v) :=  \l  \alpha x\r^{ p} \l \beta v\r^{q}
    \left| \int_0^t T_1^{-s}L_0(g,h,l)(s,x,v)\,ds\right|\nonumber\\
& = \l \alpha x\r^{p}\l \beta v\r^{q}  
    \left|\int_{0}^t  [L_0(g,h,l)](s, x+ sv, v) ds\right| \nonumber \\
& \leq\l \alpha x\r^{p}\l \beta v\r^{q} 
    \int_{0}^t \int_{\R^{9}} 
     dv_{1}\,dv_{2}\,dv_{3} ds ~
     \delta(\Sigma)\delta(\Omega) \left|g(s,x+sv, v_1)h(s,x+sv, v_2)l(s,x+sv, v_3)\right| 
    \nonumber\\
& =\l \alpha x\r^{p}\l \beta v\r^{q}  
    \int_{0}^t \int_{\R^{9}}\,dv_{1}\,dv_{2}\,dv_{3} \,ds ~ \delta(\Sigma)\delta(\Omega) \nonumber \\
&\hspace{1cm}  \left|T_1^{-s}g(s,x+s(v-v_1),v_1) \, T_1^{-s}h(s,x+s(v-v_2),v_2) \, T_1^{-s}l(s,x+s(v-v_3),v_3)\right|\\
&\leq \l \alpha x\r^{p}
    \int_{\R^{9}} \delta(\Sigma)\delta(\Omega) 
    \frac{\l \beta v\r^{q}}{\l \beta v_{1}\r^{q}\l \beta v_{2}\r^{q}\l \beta v_{3}\r^{q}} \left(\int_0^t\l\alpha x+ \alpha s u_{0,2}\r^{-p}\l \alpha x+ \alpha s u_{0,3}\r^{-p}\,ds\right)\,dv_{1}\,dv_{2}\,dv_{3},\label{a-priori 1 eq}
\end{align}
where we use the notation \eqref{manifold for equation} and where in the last inequality we used \eqref{normalization} and the fact that $\l \alpha x + \alpha s u_{0,1}\r^{-p} \le 1$.

Now, by \eqref{orthogonality}, \eqref{min estimate} from Lemma \ref{lemma on resonant manifold},  on the resonant manifold determined by $\Sigma$ and $\Omega$ we have 
\begin{equation}\label{orthogonality proof}
 u_{0,2}\cdot u_{0,3}=0,
\end{equation}
as well as
\begin{equation}\label{minimum estimate proof}
\min\{|u_{0,2}|,|u_{0,3}|\}\geq \frac{|u_{0,1}|}{2}\sqrt{1-(\hat{u}_{0,1}\cdot\hat{u}_{2,3})^2}.
\end{equation}
Thus, applying  Lemma \ref{time integral}  for $x :=\alpha x$, $\xi:=\alpha u_{0,2}$, $\eta:=\alpha u_{0,3}$, and using the above estimate, we obtain
\begin{align} \int_0^t\l \alpha x + \alpha s u_{0,2}\r^{-p} 
\l \alpha x + \alpha su_{0,3}\r^{-p}\,ds
   &\le  \frac{8p}{\alpha(p-1)}\,\frac{ \l \alpha x\r^{-p}}{ |u_{0,1}|\sqrt{1-(\hat{u}_{0,1}\cdot\hat{u}_{2,3})^2}}\label{a-priori 2 eq}.
\end{align}
Combining \eqref{a-priori 1 eq} and \eqref{a-priori 2 eq}, we obtain
\begin{align*}
I_{L_0}&(t,x, v)
\leq \frac{8p}{\alpha(p-1)}
     \int_{\R^{9}}\frac{dv_{1}\,dv_{2}\,dv_{3} ~ \delta(\Sigma)\delta(\Omega)}{|u_{0,1}|\sqrt{1-(\hat{u}_{0,1}\cdot\hat{u}_{2,
3})^2}}
    \frac{\l \beta v\r^{q}}{\l \beta v_{1}\r^{q}\l \beta v_{2}\r^{q}\l \beta v_{3}\r^{q}}.  
\end{align*}
Finally, we can use that
$$\frac{\l \beta v\r^{q}}{\l \beta v_{1}\r^{q}\l \beta v_{2}\r^{q}\l \beta v_{3}\r^{q}} 
\le \max\{\beta^q, \beta^{-3q}\} \frac{\l v\r^{q}}{\l  v_{1}\r^{q}\l v_{2}\r^{q}\l v_{3}\r^{q}},$$
together with \eqref{Uq three v} from Lemma \ref{appendix lemma on velocities weight}, to obtain 
\begin{align*}
I_{L_0}&(t,x,v)\leq \frac{8p }{\alpha(p-1)}  ~  2\pi^3 \left( \frac{1}{3} + \frac{1}{q-3}\right) ~ \max\{\beta^q, \beta^{-3q}\}. 
\end{align*}
Since $x,v,t$ were arbitrary, estimate \eqref{estimate for L eq} follows for $j=0$.

\noindent $\bullet$ {\it Proof for $L_1$.}
For any $x, v \in \R^3$ and $t\in[0,T]$, we have
\begin{align*}
I_{L_1}&(t,x, v) :=  \l  \alpha x  \r^{ p} \l \beta v\r^{q}
    \left| \int_0^t T_1^{-s}L_1(g,h,l) (s,x,v)\,ds\right|\nonumber\\
& \leq\l \alpha x\r^{p}\l \beta v\r^{q} 
    \int_{0}^t \int_{\R^{9}}  dv_{1}\,dv_{2}\,dv_{3} ds  ~ \delta(\Sigma)\delta(\Omega)    \left|g(s, x + sv, v)h(s, x + sv, v_2)l(s, x + sv, v_3)\right|   \nonumber\\
& =\l\alpha x\r^{p}\l \beta v\r^{q}  
    \int_{0}^t \int_{\R^{9}}\,dv_{1}\,dv_{2}\,dv_{3} \,ds ~
\delta(\Sigma)\delta(\Omega) \nonumber \\
&\hspace{1cm} \times \left|T_{1}^{-s}g(s, x,v) \, T_{1}^{-s}h(s, x+s(v-v_2),v_2) \,T_{1}^{-s}g(s, x+s(v-v_3),v_3)
     \right| \nonumber\\
&\leq 
    \int_{\R^{9}} 
    \,dv_{1}\,dv_{2}\,dv_{3}
    \delta(\Sigma)\delta(\Omega) 
    \left(\int_0^t\l \alpha x+ \alpha s u_{0,2}\r^{-p} \l \alpha x+ \alpha su_{0,3}\r^{-p}\,ds\right)  \frac{1}{\l \beta v_{2}\r^q  \l \beta v_{3}\r^{q}}.
\end{align*}

Applying again estimate \eqref{a-priori 2 eq} on the time integral above, the fact that $\l \beta v_{2}\r^{-q} \l \beta v_{3}\r^{-q} \le 
\max\{1, \beta^{-2q}\}
\l v_{2}\r^{-q} \l  v_{3}\r^{-q}$ together with the fact that $\l \alpha x\r^{-p}\le 1$, and  estimate \eqref{Uq two v} from Lemma \ref{appendix lemma on velocities weight}, we obtain
 \begin{align*}
 I_{L_1}(t,x, v) &\leq \frac{8p}{\alpha(p-1)}
     \max\{1, \beta^{-2q}\} \int_{\R^{9}}
   \frac{dv_{1}\,dv_{2}\,dv_{3} ~ \delta(\Sigma)\delta(\Omega)}{|u_{0,1}|\sqrt{1-(\hat{u}_{0,1}\cdot\hat{u}_{2,3})^2}\,\l  v_{2}\r^{q}\l  v_{3}\r^{q}}
    \\
&\le \frac{8p}{\alpha(p-1)}~ 2\pi^3 \left( \frac{1}{3} + \frac{1}{q-3}\right) ~\max\{1,\beta^{-2q}\},
\end{align*}
and estimate \eqref{estimate for L eq} follows for $j=1$.

\noindent $\bullet$ \textit{Proof for $ L_2$.} We note that the proof for $L_3$ is identical.
For any $x,  v\in \R^3$ and $t\in[0,T]$, we have
\begin{align}
 I_{L_2}&(t,x,v) :=\l  \alpha x\r^{ p} \l \beta v\r^{q}
    \left| \int_0^t T_1^{-s}L_2(g,h,l) (s,x,v)\,ds\right|\nonumber\\
& \leq\l \alpha x\r^{p}\l\beta v\r^{q} 
    \int_{0}^t \int_{\R^{9}}  dv_{1}\,dv_{2}\,dv_{3} ds ~ \delta(\Sigma)\delta(\Omega) \left|g(s, x+sv, v )h(s, x+sv, v_1 )h(s, x+sv, v_3 )\right| 
    \nonumber\\
& =\l \alpha x\r^{p}\l \beta v\r^{q}  
    \int_{0}^t\int_{\R^{9}}\,dv_{1}\,dv_{2}\,dv_{3} \,ds ~
    \delta(\Sigma)\delta(\Omega
    ) \nonumber \\
&\hspace{2cm} \times \left|T_{1}^{-s}g(s, x,v)|T_{1}^{-s}h(s, x+s(v-v_1),v_1)|T_{1}^{-s}l(s, x+s(v-v_3),v_3) \right| \nonumber \\
&\leq \int_{\R^{9}}  \,dv_{1}\,dv_{2}\,dv_{3}
    ~\delta(\Sigma)\delta(\Omega)    
    \frac{1}{\l \beta v_{1}\r^q  \l\beta v_{3}\r^{q}}
 \left(\int_{0}^t\l \alpha x+ \alpha s u_{0,1}\r^{-p} \,ds\right),
\end{align}
since $ \l \alpha x+ \alpha su_{0,3}\r^{-p}\le 1$. Now, by  Lemma \ref{one bracket} we have
\begin{align*}
\int_{-\infty}^{+\infty}\l \alpha x+ \alpha s u_{0,1}\r^{-p} \,ds   
\le \frac{2p}{\alpha(p-1)} | u_{0,1}|^{-1}.
\end{align*}
Therefore by using that  $\l \beta v_{3}\r^{-q} \le 1$ and the fact that $\l\beta v_1\r^{-q}\le \max\{1,\beta^{-q}\}\l v_1\r^{-q}$, we  have
\begin{align*}
 I_{L_2}(t,x, v) &\leq \frac{2p}{\alpha(p-1)}
     \max\{1,\beta^{-q}\} \int_{\R^{9}} \frac{\delta(\Sigma)\delta(\Omega)  }{ |u_{0,1}| \l v_1\r^{q} }
    \,dv_{1}\,dv_{2}\,dv_{3}\nonumber\\
& \le  \frac{8p}{\alpha(p-1)} ~ \pi^2 \left(\frac{1}{3} + \frac{1}{q-3}\right) ~ \max\{1,\beta^{-q}\},
\end{align*}
where in the last inequality we used Lemma \ref{delta convolution lemma} for $d=3$. Estimate \eqref{estimate for L eq} then follows for $j=2$.
\end{proof}

Now, using the above a-priori estimate in conjunction with the contraction mapping principle, we can prove the global well-posedness of \eqref{KWE} for arbitrary signed data which are polynomially close to vacuum. In order to guarantee non-negativity of the solution when the initial data are non-negative, one needs a more refined argument  taking advantage of the monotonicity properties of the equation. This can be achieved by employing a classical tool from the kinetic theory of particles, namely the Kaniel-Shinbrot iteration \cite{kash78, ilsh84, beto85, to86, to88,  al09, alga09}. Recently in \cite{am24}, this technique was applied for the first time in the context of wave turbulence by the first author of this paper, who addressed the problem for exponentially decaying initial data.

\begin{proof}[Proof of Theorem \ref{thm-kwe is well posed}]
In order to obtain existence, uniqueness, non-negativity and stability of solutions to \eqref{KWE}, we can essentially follow the strategy of the proofs in \cite{am24} (using  Proposition \ref{a-priori estimate equation} instead of the exponential type a priori estimate of \cite{am24} when needed), so we omit details of the proof,  and show the setup only. In particular, let us define the mapping $\Phi$ as follows
\begin{align}
    \Phi(g(t)) = f_0 + \int_0^t T_1^{-s} \mathcal{C}[T_1^s g](s) ds.
\end{align}
Then using Proposition \ref{a-priori estimate equation} it can be shown that $\Phi: B^M_{X_{p,q,\alpha,\beta,T}}\mapsto B^M_{X_{p,q,\alpha,\beta,T}}$, where
\begin{align}
 B^M_{X_{p,q,\alpha,\beta,T}} := 
 \{h\in X_{p,q,\alpha,\beta,T} \text{ \, with \, } ||| h |||_{p,q,\alpha, \beta, T} \le M \},
\end{align}
is a contraction in $B^M_{X_{p,q,\alpha,\beta,T}}$. Hence, there exists a unique fixed point $g \in  B^M_{X_{p,q,\alpha,\beta,T}}$ such that
\begin{align}
    g(t) = \Phi(g(t)).
\end{align}
By letting $f(\cdot) = T_1^{(\cdot)} g(\cdot)$, this is equivalent to
\begin{align}
T_1^{-t} f(t) = \Phi(T_1^{-t} f(t)) = f_0 + \int_0^t T_1^{-s} \mathcal{C}[f](s) dx,
\end{align}
and we also have $T_1^{-t} f \in B^M_{X_{p,q,\alpha,\beta,T}}$ as claimed. This completes proof of well-posedness.

It remains to verify the conservation laws \eqref{conservation of mass KWE}-\eqref{conservation of energy KWE} for the constructed solution $f$. For this, we define the space
$$L_v^{1,\ell}=\left\{g:\R^3\to\R \text{ measurable such that }\int_{\R^3}\l v\r^\ell g(v)\,dv<\infty\right\},$$
where $\ell\ge 0$.

Assume that $p>3$ and $q>4+i$, where  $i\in\{0,1,2\}$. By \eqref{condition on uniqueness equation}, for any $t\in [0,T]$ and a.e. $x,v\in\R^3$, we have
\begin{equation}\label{pointwise bound}\l \beta v\r^{1+i}f(t,x,v)\le M\l \alpha(x-vt)\r^{-p}\l \beta v\r^{1+i-q}\leq M\l \beta v\r^{1+i-q}.
\end{equation}
Since $q>4+i$, integrating \eqref{pointwise bound} in velocity we  obtain that $f(t,x,v)\in L^{1,1+i}_v$, for any $t\in[0,T]$ and a.e. $x\in\R^d$.
Therefore, by Lemma \ref{delta lemma},  the weak form \eqref{weak form} is valid for any test function $\phi$ with $|\phi(v)|\le C \l v\r^i$.  Hence,   using \eqref{weak form} and the resonant conditions, we obtain
\begin{align}
&\int_{\R^3} \mathcal{C}[f]\phi(v)\,dv=0, \quad\forall t\in[0,T],\text{  a.e. $x\in \R^3$} \label{collisional averaging use}, 
\end{align}
where $\phi=1$ if $i=0$, $\phi\in\{1,v\}$ if $i=1$, and $\phi\in\{1,v,|v|^2\}$ if $i=2$.

Additionally $f\in L^\infty([0,T],L^1_x L^{1,1+i}_v)$, which can be seen by integrating the first inequality of \eqref{pointwise bound} in $x,v$ and taking supremum in time:
\begin{align}
\sup_{t\in[0,T]}\int_{\R^{6}}\l \beta v\r^{1+i}f(t,x,v) \,dx\,dv&\leq M\sup_{t\in[0,T]}\int_{\R^3}\left(\int_{\R^3}\l \alpha(x-vt)\r^{-p}\,dx\right)\l \beta v\r^{1+i-q}\,dv\nonumber\\
& =M \alpha^{-3} \beta^{-3} \left(\int_{\R^3}\l x'\r^{-p}\,dx'\right)\left(
\int_{\R^3}\l v'\r^{1+i-q}\,dv'\right)<\infty,\label{KWE solution moments}
\end{align}
since $p>3$ and $q>4+i$. By Lemma \ref{delta lemma}, we also obtain that $Q(f,f)\in L^\infty([0,T],L^1_x L^{1,i}_v)$.

Now let $\phi=1$ if $i=0$, $\phi\in\{1,v\}$ if $i=1$, and $\phi\in\{1,v,|v|^2\}$ if $i=2$.
We integrate \eqref{def-kwe mild solution equation} in $x,v$ and use Fubini's theorem and \eqref{collisional averaging use}, to obtain that for any $t\in[0,T]$ we have 
\begin{align}
\int_{\R^{6}}f(t,x+tv,v)\phi(v)\,dx\,dv
&=\int_{\R^{6}}f_0(x,v)\phi(v)\,dxdv+\int_0^t\int_{\R^3}\int_{\R^3} \mathcal{C}[f](s,x+sv,v)\phi(v)dxdvds\nonumber\\
&=\int_{\R^{6}}f_0(x,v)\phi(v)\,dx\,dv+\int_0^t\int_{\R^3}\int_{\R^3} \mathcal{C}[f](s,x,v)\phi(v)\,dx\,dv\,ds\nonumber\\
&=\int_{\R^{6}}f_0(x,v)\phi(v)\,dx\,dv\nonumber.
\end{align}
Thus for any $t\in[0,T]$, we have $$\int_{\R^{6}}f(t,x,v)\phi(v)\,dx\,dv=\int_{\R^{6}}f(t,x+vt,v)\,\phi (v) \,dx\,dv=\int_{\R^{6}}f_0(x,v)\phi (v)\,dx\,dv,$$
and \eqref{conservation of mass KWE}-\eqref{conservation of energy KWE} follow.

\end{proof}

\section{Global well-posedness of wave kinetic hierarchy}\label{sec-KWH}

In this section we construct the unique global in time mild solution to the wave kinetic hierarchy \eqref{kwh} for admissible initial data and for a range of values of  the parameter $\mu$. The construction is inspired by a similar construction of mild solutions for the Boltzmann hierarchy in \cite{ammipata24}. The strategy consists of the following steps. First, since the initial data will be assumed to be admissible, we will be able to employ a Hewitt-Savage representation \cite{hs55} tailored to our norms, to express such  initial data as a convex combination of tensorized states under an appropriate probability measure $\pi$. Then, each element in the support of the measure $\pi$ can be treated as initial data to the wave kinetic {\it equation} for which we have a mild solution. Finally, by taking the same convex combination of these tensorized solutions to the wave kinetic {\it equation} under the same probability measure $\pi$, we will prove that one obtains a mild solution to the wave kinetic {\it hierarchy}. Finally, we will utilize the uniqueness result (Theorem \ref{thm-uniqueness}) to conclude that that mild solution is unique.

\subsection{Proof of existence for the wave kinetic hierarchy}
In this section we prove the existence part of Theorem \ref{thm-gwp for kwh}. In particular we claim the following.  
\begin{theorem}[Existence of solutions for the wave kinetic hierarchy]\label{thm-existence for kwh}
 Consider the wave kinetic hierarchy \eqref{kwh} in dimension $d=3$.
Let $T>0$, $p>1$, $q>3$, $\alpha, \beta>0$  and let $\mu\in\R$ be such that
        $e^{2\mu}>32C_{p,q,\alpha,\beta},$
    where $C_{p,q,\alpha,\beta}$ is given by \eqref{constant C}.
Consider  admissible initial data $F_0=(f_0^{(k)})_{k=1}^\infty\in \mathcal{A}\cap \mathcal{X}_{p,q,\alpha,\beta,\mu'}^\infty$, where $\mu'=\mu+\ln 2$. Then, there exists a  non-negative mild $\mu$-solution $F=(f^{(k)})_{k=1}^\infty$ of the wave kinetic hierarchy \eqref{kwh}. In addition, this solution  satisfies the estimate
\begin{equation}\label{stability estimate hierarchy existence}
|||\mathcal{T}^{-(\cdot)}F(\cdot)|||_{p,q,\alpha,\beta,\mu,T}\leq  1.
\end{equation}
\end{theorem}

We start by stating the Hewitt-Savage theorem that is tailored to our norms. It is exactly the same as Proposition 4.4. in \cite{ammipata24}, but we include it here so  that the paper is self-contained.
Similar versions of this theorem can be found in, for example,  \cite[Proposition 6.1.3]{gastte13}, \cite[Theorem 2.6]{foleso18}, \cite{rost22}, \cite{deha22}.

\begin{proposition}[Hewitt-Savage] \label{hewitt-savage}
    Suppose  $G=(g^{(k)})_{k=1}^\infty$ is admissible in the sense of Definition \ref{admissible data}. Then, there exists a unique Borel probability measure $\pi$ on the set of probability measures $\mathcal{P}$, where
    \begin{equation} \label{P definition}
\mathcal{P}=\left\{h\in L^1(\R^{2d}): h\geq 0,\quad\int_{\R^{2d}}h(x,v)\,dx\,dv=1\right\},
\end{equation}
     such that 
    \begin{equation}\label{representation of marginal}
         g^{(k)} = \int_{\mathcal{P}} h^{\otimes k} d\pi(h), \qquad  \forall \, k \in \N.
    \end{equation}
    If additionally $G\in \mathcal{X}_{p,q,\alpha,\beta,\mu'}^\infty$,
 for some $p,q>1$, $\alpha,\beta>0$ and $\mu'\in\R$, then
    \begin{equation}\label{support inclusion}
       \support(\pi)\subseteq \left\{ h\in \mathcal{P} : \Vert h \Vert_{p,q,\alpha,\beta} \leq e^{-\mu'} \right\}.
    \end{equation}
\end{proposition}

\begin{remark}\label{ball}
    We note that representation \eqref{representation of marginal}, and the support condition \eqref{support inclusion}   imply that $\mathcal{A}\cap \mathcal{X}_{p,q,\alpha,\beta,\mu'}^\infty=\mathcal{A}\cap B_{\mathcal{X}_{p,q,\alpha,\beta,\mu'}^\infty}$, where $B_{\mathcal{X}_{p,q,\alpha,\beta,\mu'}^\infty}$ denotes the unit ball of $\mathcal{X}_{p,q,\alpha,\beta,\mu'}^\infty$.
\end{remark}

We are now ready to prove Theorem \ref{thm-existence for kwh} in a similar manner to the proof of the analogous result for the Boltzmann hierarchy (see proof of \protect{\cite[Theorem 2.1]{ammipata24}}).

\begin{proof}[Proof of Theorem \ref{thm-existence for kwh}]  Let $\mu$ be such that $e^{2\mu}>32C_{p,q,\alpha,\beta}$, let $\mu' = \mu+\ln 2$, and consider an initial datum $F_0=(f_0^{(k)})_{k=1}^\infty\in \mathcal{A}\cap \mathcal{X}_{p,q,\alpha,\beta,\mu'}^\infty$. By the Hewitt-Savage theorem (Proposition \ref{hewitt-savage}), there exists  a Borel probability measure $\pi$ on $\mathcal{P}$ such that 
    \begin{equation}
     f_0^{(k)} = \int_{\mathcal{P}} h_0^{\otimes k} d\pi(h_0),
    \end{equation}
    and 
    \begin{equation}\label{support proof}
        \text{supp}(\pi)\subseteq \left\{ h_0 \in \mathcal{P} : \Vert h_0\Vert_{p,q,\alpha,\beta} \leq e^{-\mu'} \right\}.
    \end{equation}
Therefore,   for $\pi$-almost any $h_0 \in \mathcal{P}$, we have that $\Vert h_0 \Vert_{p,q,\alpha,\beta}\leq e^{-\mu'}= \frac{e^{-\mu}}{2}$.
Let $M=e^{-\mu}$ and note that then $\Vert h_0 \Vert_{p,q,\alpha,\beta} \le \frac{M}{2}$, and that due to the assumption $e^{2\mu}>32C_{p,q,\alpha,\beta}$, we have that $M = e^{-\mu} < (32 C_{p,q,\alpha,\beta})^{-1/2} <
(24 C_{p,q,\alpha,\beta})^{-1/2}$. Therefore, we can apply Theorem \ref{thm-kwe is well posed} with $M=e^{-\mu}$ and initial data   $h_0$ to conclude that there is  a mild solution $h(t)$ to the wave kinetic equation corresponding to the initial datum $h_0$. By \eqref{bound wrt initial data equation}, we have that this solution  satisfies 
   \begin{equation}\label{pf-bound on h}
   ||T_1^{-t}h(t)\|_{p,q, \alpha, \beta}\leq 2\|h_0\|_{p,q,\alpha,\beta} \le e^{-\mu},\quad\forall\, t\in[0,T].    
   \end{equation}   
Thanks to the continuity with respect to initial data  \eqref{stability estimate},  given $t\in[0,T]$, the map $h_0\mapsto h(t)$ is continuous and thus Borel measurable.

Now we construct an infinite sequence of functions, which we will show will be a mild solution of the wave kinetic hierarchy \eqref{kwh}. Namely, we define $F=(f^{(k)})_{k=1}^\infty$, by
    \begin{align}\label{kwh solution}
        f^{(k)}(t) : = \int_{\mathcal{P}} h(t)^{\otimes k} d\pi(h_0),\quad t\in[0,T],\,\, k\in\N.
    \end{align}

For any $k\in\N$ and $t\in[0,T]$,  we have
\begin{align*}
        e^{\mu k}\Vert T^{-t}_k f^{(k)}(t)\Vert_{k,p,q,\alpha,\beta} &\leq  e^{\mu k}\int_{\mathcal{P}} \Vert T^{-t}_k h^{\otimes k}(t)\Vert_{k,p,q,\alpha,\beta} d\pi(h_0)\nonumber \\
        & =  e^{\mu k}\int_{\mathcal{P}} \Vert T^{-t}_1 h(t)\Vert_{p,q,\alpha,\beta}^k d\pi(h_0) \nonumber \\
        &\leq 1,
    \end{align*}
where the equality follows from \eqref{tensorization of transport norm} and the last inequality follows from \eqref{pf-bound on h}. Since this estimate is true for any  $k\in\N$ and $t\in[0,T]$, the estimate \eqref{stability estimate hierarchy existence} follows, which also  implies that $\mathcal{T}^{-(\cdot)}F(\cdot)\in \mathcal{X}_{p,q,\alpha,\beta,\mu,T}^\infty$.
Now that we know that $F$ is in the right space, a standard computation shows that $F$, defined via \eqref{kwh solution}, is a mild $\mu$-solution to the wave kinetic hierarchy \eqref{kwh}, corresponding to initial data $F_0\in \mathcal{X}_{p,q,\alpha,\beta,\mu'}^\infty\subset \mathcal{X}_{p,q,\alpha,\beta,\mu}^\infty $. 
This solution is non-negative since $h(t)$ in \eqref{kwh solution} are solutions to the wave kinetic equation with non-negative initial data $h_0$.

\end{proof}

\subsection{Proof of uniqueness of mild solutions to the wave kinetic hierarchy}
The goal of this section is to prove uniqueness of solutions to the wave kinetic hierarchy stated as follows:
\begin{theorem}[Uniqueness of solutions to the wave kinetic hierarchy]\label{thm-uniqueness}
    Consider the wave kinetic hierarchy \eqref{kwh} in dimension $d=3$. Let $T>0$ and assume that $p>1, q>3$ and $\alpha, \beta >0$. Let $\mu \in \R$ be such that $e^{2\mu} > 32 C_{p,q,\alpha,\beta}$ for $C_{p,q,\alpha,\beta}$ as in \eqref{constant C}. Let $F_0 = (f_0^{(k)})_{k=1}^\infty \in X^\infty_{p,q,\alpha,\beta}$ and assume  $F = (f^{(k)})_{k=1}^\infty \in X^\infty_{p,q,\alpha,\beta,T}$ is a mild $\mu$-solution to  wave kinetic hierarchy \eqref{kwh} corresponding to the initial data $F_0$. Then $F$ is unique.

\end{theorem}

Since the wave kinetic hierarchy in linear, it suffices to show that if $F_0=0$ and $F$ is a mild solution, then $F=0$. In order to motivate the proof, let us start by recalling that a mild solution $F = (f^{(k)})_{k=1}^\infty $ to the wave kinetic hierarchy \eqref{kwh} is given by
\begin{align*}
  T^{-t}_k f^{(k)}(t) =  f_0^{(k)} + \int_0^{t} T_k^{-s }\C^{k+2} f^{(k+2)}(s) ds, \quad \forall k \in \N, \,\, \forall\,t\in[0,T].
\end{align*}
When initial data is zero ($F_0=0$), one can apply this formula iteratively $n\in\N$ times to obtain
\begin{align}\label{iterated Duhamel}
    T^{-t}_k f^{(k)}(t)
	 &= \int_0^{t}\int_{0}^{t_{k+2}} \cdots \int_0^{t_{k+2n-2}}  dt_{k+2n} \cdots dt_{k+4} dt_{k+2} \nonumber \\
     & \qquad T_k^{- t_{k+2}} \C^{k+2} T_{k+2}^{t_{k+2}- t_{k+4}} \C^{k+4} \cdots T^{t_{k+2n-2} - t_{k+2n}}_{k+2n-2} \C^{k+2n} f^{(k+2n)}(t_{k+2n}).
\end{align}
In order to prove that $F=0$, it will suffice to  show that $ \left\| T^{-t}_k f^{(k)}(t) \right\|_{p,q,\alpha,\beta}\lesssim C^n$, where $C<1$. Then by letting $n\to \infty$ one can conclude that for each $k\in \N$, $f^{(k)} =0$, thus completing the proof. Recall that each $\C^{k+2j}$ is a sum of $k+2j-2$ operators, and thus the iterated Duhamel formula \eqref{iterated Duhamel} contains a factorial number of terms. For each of them, we will use  an (iterated) a priori estimate (see Proposition \ref{prop-iterated a priori}). This fact alone is not enough to obtain the needed geometric growth estimate since the number of terms is factorial. This is where a board game argument inspired by \cite{klma08, chpa11} will come into play and will allow us to rearrange the factorial number of terms in \eqref{iterated Duhamel} in an exponential number of equivalence classes, and the sum over each equivalence class will be bounded by a constant, thus allowing us to obtain the geometric growth estimate $ \left\| T^{-t}_k f^{(k)}(t) \right\|_{p,q,\alpha,\beta}\lesssim C^n$.

We have organized this section into three subsections. The first one establishes the (iterated) a priori estimate. The second subsection discusses the combinatorial board game argument. Finally, in the third subsection, we combine the a priori estimate and the board game argument to prove the uniqueness of  solutions to the wave kinetic hierarchy as stated in Theorem \ref{thm-uniqueness}.

Throughout this section, we will use the following notation
\begin{align}\label{u_ij}
    u_{i,j}=v_i-v_j.
\end{align}
Also  for any vector $v \in \R^3$, we denote by $\hat{v}$ its unit vector
\begin{align}
    \hat{v} = \frac{v}{|v|}.
\end{align}

\subsubsection{A priori estimate}\label{section - a priori estimates} 

The following estimate is the hierarchy level analogue of the nonlinear a priori estimate presented in Proposition \ref{a-priori estimate equation}.

\begin{proposition} \label{a priori estimate}
 Let $T>0$, $p>1$ and $q>3$. Let $\C^\lambda_{j,k+2}$, where $\lambda \in \{ L_0, L_1, L_2, L_3\}$, be defined in \eqref{CM}-\eqref{CL3}.
 Then for any $t\in[0,T]$, $k\in\N$, $j\in\{1,\cdots, k\}$, and $\lambda \in \{ L_0, L_1, L_2, L_3\}$ we have
\begin{equation}\label{estimate for M}
\left\|\int_0^t T_k^{-s}\C^{\lambda}_{j,k+2}g^{(k+2)}(s)\,ds\right\|_{k, p,q,\alpha,\beta}\leq  C_{p,q,\alpha,\beta}|||T_{k+2}^{-(\cdot)}g^{(k+2)}(\cdot)|||_{k+2,p,q,\alpha,\beta,T}, 
\end{equation}
where $C_{p,q,\alpha,\beta}$ is the constant appearing in \eqref{constant C}.
\end{proposition}
\begin{proof} 
The proof of the estimate \eqref{estimate for M} is analogous to the proof of the nonlinear estimate \eqref{estimate for L eq} at the level of the equation. As a result, we will only present the proof for $\lambda=L_0$ to demonstrate this. The other cases are proved in a similar spirit to the proof of Proposition \ref{a-priori estimate equation}.

Fix $k\in\N$, $j\in\{1,\cdots, k\}$ and $\lambda = L_0$. Recall notation introduced in \eqref{u_ij}.
For any $X_k,V_k\in\R^{3k}$ we have
\begin{align}
I_{L_0}&(X_k, V_k) :=  \l\l  \alpha X_k\r\r^{ p} \l\l \beta V_k\r\r^{q}
    \left| \int_0^t T_k^{-s}\C_{j,k+2}^{L_0} g^{(k+2)}(s,X_k,V_k)\,ds\right|\nonumber\\
& = \l\l \alpha X_k\r\r^{p}\l\l \beta V_k\r\r^{q}  
    \left|\int_{0}^t  [\C^{L_0}_{j,k+2} g^{(k+2)}](s, X_k+ sV_k, V_k) ds\right| \nonumber \\
& \leq\l\l \alpha X_k\r\r^{p}\l\l \beta V_k\r\r^{q} 
    \int_{0}^t \int_{\R^{9}} 
     dv_{k+1}\,dv_{k+2}\,dv_{k+3} ds ~
     \delta(\Sigma_{j,k+2})\delta(\Omega_{j,k+2}) \nonumber \\
& \hspace{1cm}   \times \left|g^{(k+2)}(s, X_k + sV_k, x_j + sv_j, x_j+sv_j, ~V_k^{j,v_{k+1}} ,v_{k+2},v_{k+3})\right| 
    \nonumber\\
& =\l\l \alpha X_k\r\r^{p}\l\l \beta V_k\r\r^{q}  
    \int_{0}^t \int_{\R^{9}}\,dv_{k+1}\,dv_{k+2}\,dv_{k+3} \,ds ~ \delta(\Sigma_{j,k+2})\delta(\Omega_{j,k+2}) \nonumber \\
&\hspace{1cm} \left|T_{k+2}^{-s}g^{(k+2)}(s, X_k + s(V_k-V_k^{j,v_{k+1}}),\, x_j + su_{j,k+2}, \, x_k+su_{j,k+3}, ~
    V_{k}^{j,v_{k+1}} ,v_{k+2},v_{k+3}) \right| \nonumber \\
&\leq \l \alpha x_j\r^{p}|||T_{k+2}^{-(\cdot)}g^{(k+2)}(\cdot)|||_{k+2,p,q,\alpha,\beta,T} 
    \int_{\R^{9}} \delta(\Sigma_{j,k+2})\delta(\Omega_{j,k+2}) 
    \frac{\l \beta v_j\r^{q}}{\l \beta v_{k+1}\r^{q}\l \beta v_{k+2}\r^{q}\l \beta v_{k+3}\r^{q}}\nonumber \\
&\hspace{1cm}\times \left(\int_0^t\l\alpha x_j+ \alpha s u_{j,k+2}\r^{-p}\l \alpha x_j+ \alpha s u_{j,k+3}\r^{-p}\,ds\right)\,dv_{k+1}\,dv_{k+2}\,dv_{k+3},\label{a-priori 1}
\end{align}
where in the last inequality we used that $\l \alpha x_j + \alpha s u_{j, k+1}\r^{-p} \le 1$.

Now, by \eqref{orthogonality}, \eqref{min estimate} from Lemma \ref{lemma on resonant manifold},  on the resonant manifold determined by $\Sigma_{j,k+2}$ and $\Omega_{j,k+2}$ we have 
\begin{equation*}
 u_{j,k+2}\cdot u_{j,k+3}=0,
\end{equation*}
as well as
\begin{equation}\label{minimum estimate proof hierarchy}
\min\{|u_{j,k+2}|,|u_{j,k+3}|\}\geq \frac{|u_{j,k+1}|}{2}\sqrt{1-(\hat{u}_{j,k+1}\cdot\hat{u}_{k+2,k+3})^2}.
\end{equation}

Applying  Lemma \ref{time integral}  for $x =\alpha x_j$, $\xi=\alpha u_{j,k+2}$, $\eta=\alpha u_{j,k+3}$, and using the above estimate \eqref{minimum estimate proof hierarchy}, yields
\begin{align} \int_0^t\l \alpha x_j + \alpha s u_{j,k+2}\r^{-p} 
\l \alpha x_j + \alpha su_{j,k+3}\r^{-p}\,ds
   &\le  \frac{8p}{\alpha(p-1)}\,\frac{ \l \alpha x_j\r^{-p}}{ |u_{j,k+1}|\sqrt{1-(\hat{u}_{j,k+1}\cdot\hat{u}_{k+2,k+3})^2}}\label{a-priori 2}.
\end{align}

Combining \eqref{a-priori 1} and \eqref{a-priori 2}, we obtain
\begin{align*}
I_{L_0}&(X_k, V_K)
\leq \frac{8p}{\alpha(p-1)}
    |||T_{k+2}^{-(\cdot)}g^{(k+2)}(\cdot)|||_{k+2,p,q,\alpha,\beta,T} \int_{\R^{9}}
    dv_{k+1}\,dv_{k+2}\,dv_{k+3} ~ \delta(\Sigma_{j,k+2})\delta(\Omega_{j,k+2}) \nonumber\\
&\hspace{4cm} \times   \frac{1}{|u_{j,k+1}|\sqrt{1-(\hat{u}_{j,k+1}\cdot\hat{u}_{k+2,k+3})^2}}
    \frac{\l \beta v_j\r^{q}}{\l \beta v_{k+1}\r^{q}\l \beta v_{k+2}\r^{q}\l \beta v_{k+3}\r^{q}}.
\end{align*}
Finally, we can use that
$$\frac{\l \beta v_j\r^{q}}{\l \beta v_{k+1}\r^{q}\l \beta v_{k+2}\r^{q}\l \beta v_{k+3}\r^{q}} 
\le \max\{\beta^q, \beta^{-3q}\} \frac{\l v_j\r^{q}}{\l  v_{k+1}\r^{q}\l v_{k+2}\r^{q}\l v_{k+3}\r^{q}},$$
together with Lemma \ref{appendix lemma on velocities weight}, to obtain 
\begin{align*}
I_{L_0}&(X_k, V_K)\leq \frac{8p }{\alpha(p-1)}  ~  2\pi^3 \left( \frac{1}{3} + \frac{1}{q-3}\right) ~ \max\{\beta^q, \beta^{-3q}\} ~  |||T_{k+2}^{-(\cdot)}g^{(k+2)}(\cdot)|||_{k+1,p,q,\alpha,\beta,T}.
\end{align*}
Since $t, X_k,V_k$ were arbitrary, estimate \eqref{estimate for M} follows. 
\end{proof}

We next derive an iterated estimate by recursively applying Proposition \ref{a priori estimate}. We start by recalling the notation introduced in \eqref{gain}-\eqref{loss}.

\begin{proposition}\label{prop-iterated a priori}
 Let $T>0$, $p>1, q>3$ and $k,n \in \N$ with $n\ge 2$. For each $\ell \in \{1, 2, \dots, n\}$, let $j_{k+2\ell}$ be a number in the set $\{1,2, \dots, k+2\ell-2\}$ and let $\pi_{k+2\ell} \in \{+,-\}$. Then we have
 \begin{align}\label{iterated a priori}
&\Bigg\|\int_{[0,T]^n} 
 T_k^{- t_{k+2}} \C^{\pi_{k+2}}_{j_{k+2},k+2} T_{k+2}^{t_{k+2}- t_{k+4}} \C^{\pi_{k+4}}_{j_{k+4}, k+4} \cdots \nonumber\\
& \hspace{1cm} \cdots T^{t_{k+2n-2} - t_{k+2n}}_{k+2n-2} \C^{\pi_{k+2n}}_{j_{k+2n},k+2n} f^{(k+2n)}(t_{k+2n}) ~ dt_{k+2n} \cdots dt_{k+4} dt_{k+2} \Bigg\|_{k,p,q,\alpha,\beta} \nonumber\\
&\le \left( 2C_{p,q,\alpha,\beta} \right)^n
    \Big|\Big|\Big| T_{k+2n}^{-(\cdot)} f^{(k+2n)}\Big|\Big|\Big|_{k+2n, p,q,\alpha,\beta, T},
 \end{align}
 where $C_{p,q,\alpha,\beta}$ is given in \eqref{constant C}.
\end{proposition}

\begin{proof} Let
\begin{align*}
I := &\Bigg\|\int_{[0,T]^n} 
 T_k^{- t_{k+2}} \C^{\pi_{k+2}}_{j_{k+2},k+2} T_{k+2}^{t_{k+2}- t_{k+4}} \C^{\pi_{k+4}}_{j_{k+4}, k+4} \cdots \\
& \hspace{1cm} \cdots T^{t_{k+2n-2} - t_{k+2n}}_{k+2n-2} \C^{\pi_{k+2n}}_{j_{k+2n},k+2n} f^{(k+2n)}(t_{k+2n}) ~ dt_{k+2n} \cdots dt_{k+4} dt_{k+2} \Bigg\|_{k,p,q,\alpha,\beta}\\
& = \Bigg\| \int_0^T  T_k^{- t_{k+2}} \C^{\pi_{k+2}}_{j_{k+2},k+2} ~ G^{(k+2)} (t_{k+2}) ~ dt_{k+2}\Bigg\|_{k,p,q,\alpha,\beta}
\end{align*}
where
\begin{align*}
     G^{(k+2)} (t_{k+2}) = \int_{[0,T]^{n-1}}  & dt_{k+2n} \cdots dt_{k+4}  \\
&T_{k+2}^{t_{k+2}- t_{k+4}} \C^{\pi_{k+4}}_{j_{k+4}, k+4} \cdots T^{t_{k+2n-2} - t_{k+2n}}_{k+2n-2} \C^{\pi_{k+2n}}_{j_{k+2n},k+2n} f^{(k+2n)}(t_{k+2n}).
\end{align*}
By Proposition \ref{a priori estimate}, we have
\begin{align*}
I & \le 2C_{p,q,\alpha,\beta} \Big|\Big|\Big| T_{k+2}^{-(\cdot)} G^{(k+2)}\Big|\Big|\Big|_{k+2, p,q,\alpha,\beta, T}\\
& = 2C_{p,q,\alpha,\beta} \sup_{t_{k+2} \in[0,T]}
\Bigg\| 
T_{k+2}^{-t_{k+2}} \int_{[0,T]^{n-1}}   dt_{k+2n} \cdots dt_{k+4}  \\
&\hspace{2.5cm} T_{k+2}^{t_{k+2}- t_{k+4}} \C^{\pi_{k+4}}_{j_{k+4}, k+4} \cdots T^{t_{k+2n-2} - t_{k+2n}}_{k+2n-2} \C^{\pi_{k+2n}}_{j_{k+2n},k+2n} f^{(k+2n)}(t_{k+2n})
\Bigg\|_{k+2, p,q,\alpha,\beta}\\
& =  2C_{p,q,\alpha,\beta} 
\Bigg\| 
\int_{[0,T]^{n-1}}   dt_{k+2n} \cdots dt_{k+4}  \\
&\hspace{3cm} T_{k+2}^{- t_{k+4}} \C^{\pi_{k+4}}_{j_{k+4}, k+4} \cdots T^{t_{k+2n-2} - t_{k+2n}}_{k+2n-2} \C^{\pi_{k+2n}}_{j_{k+2n},k+2n} f^{(k+2n)}(t_{k+2n})
\Bigg\|_{k+2, p,q,\alpha,\beta}.
\end{align*}
Repeating these calculations $(n-1)$ more times, we obtain
\begin{align*}
    I \le \left( 2C_{p,q,\alpha,\beta} \right)^n
    \Big|\Big|\Big| T_{k+2n}^{-(\cdot)} f^{(k+2n)}\Big|\Big|\Big|_{k+2n, p,q,\alpha,\beta, T}.
\end{align*}
\end{proof}

\subsubsection{Combinatorial board game argument}
In this section, inspired by \cite{klma08} and \cite{chpa11}, we devise a combinatorial board game argument which, together with the a priori estimates from the previous section, will allow us to prove uniqueness of solutions to the wave kinetic hierarchy \eqref{kwh}. 
As indicated earlier in this section, the wave kinetic hierarchy is linear, and thus uniqueness boils down to proving that the solution corresponding to zero initial data is zero. Recall that such a solution has the expansion \eqref{iterated Duhamel}, and that the integrand in the iterated time integral has a factorial number of terms since each  operator $\C^{k+2\ell}$ is a sum of $k+2\ell-2$ operators. The goal of this section is to find a way to reorganize the integral \eqref{iterated Duhamel} in such a way that allows us to estimate it with an exponential number of terms instead. In order to achieve this, we start by introducing some notation, starting with the notation for the integrand in \eqref{iterated Duhamel}.

\begin{definition}
Let $k,n \in \N$ with $n\ge 2$,  and let $\t_{k+2n} = (t_{k+2}, t_{k+4}, \dots , t_{k+2n}) \in \R_+^n$. We define the operator $J_{n,k}$ by
\begin{align}
J_{n,k}(\t_{k+2n}) f^{(k+2n)} 
= T_k^{- t_{k+2}} \C^{k+2} T_{k+2}^{t_{k+2}- t_{k+4}} \C^{k+4} \cdots T^{t_{k+2n-2} - t_{k+2n}}_{k+2n-2} \C^{k+2n} f^{(k+2n)}(t_{k+2n}).
\end{align}
\end{definition}

By \eqref{C^k}, each  $\C^{k+2\ell}$ is a sum of $k+2\ell -2$ terms, so we can write the $J_{n,k}$ operator as a sum
\begin{align}
    J_{n,k}(\t_{k+2n}) = \sum_{\mu \in M_{n,k}}  J_{n,k}(\t_{k+2n}, \mu),
\end{align}
where
\begin{align}\label{set M}
    M_{n,k} &= \Big\{\mu: \{k+2, k+4, \dots, k+2n\} \to \{1, 2, \dots, k+2n-2 \}, \text{ with } \, \forall j \,\, \mu(j) < j-1  \Big\}
\end{align}
and 
\begin{align}\label{J mu}
J_{n,k}(\t_{k+2n}, \mu) f^{(k+2n)}
&:= T_k^{- t_{k+2}} \C_{\mu(k+2),k+2} T_{k+2}^{t_{k+2}- t_{k+4}} \C_{\mu(k+4), k+4} \cdots \nonumber \\
& \hspace{2cm} \cdots T^{t_{k+2n-2} - t_{k+2n}}_{k+2n-2} \C_{\mu(k+2n),k+2n} f^{(k+2n)}(t_{k+2n}).
\end{align}

We next define operator $\I_{n,k}$ as a time integral of the operator $J_{n,k}$. In what follows, we use the following notation  for the group of all permutations of elements $\{k+2, \dots,k+ 2n\}$:
\begin{align}\label{S}
S_{n,k} = S(\{k+2, \dots,k+ 2n\}).
\end{align}

\begin{definition}
\label{I integrals}
Let $k,n \in \N$ with $n\ge 2$. For each  $(\mu, \sigma) \in M_{n,k}\times S_{n,k}$
define the operator  $\mathcal{I}_{n,k}(\mu,\sigma)$  by the expression 
\begin{equation}\label{single iterated Duhamel}
    \mathcal{I}_{n,k}(\mu,\sigma) : = \int_{t \geq t_{\sigma(k+2)} \geq t_{\sigma(k+4)} \geq \dots \geq t_{\sigma(k+2n)}\geq 0} J_{n,k}( \t_{k+2n};\mu) dt_{k+2n} dt_{k+2n-2} \dots dt_{k+2}.
\end{equation}
\end{definition}
Note that it is equivalent to write 
\begin{equation}
    \mathcal{I}_{n,k}(\mu,\sigma)  = \int_{t\geq t_{k+2} \geq t_{k+4} \geq \dots \geq t_{k+2n}\geq 0} J_{n,k}( \sigma^{-1}(\t_{k+2n});\mu) dt_{k+2n} dt_{k+2n-2} \dots dt_{k+2},
\end{equation}
where  
\begin{equation}
    \sigma^{-1}(\t_{k+2n}) : = ( t_{\sigma^{-1}(k+2)}, \dots , t_{\sigma^{-1}(k+2n)}).
\end{equation}

The operator $\I_{n,k}$ can be represented on a $(k+2n-2) \times n$ board with carved in  names $\C_{i,j}$, with $1\le i \le j-2$, for $j\in\{k+2, k+4, ..., k+2n\}$, arranged as in \eqref{board}. The board also has an extra top row that keeps track of the order of times in the operator $\I_{n,k}$. For each $(\mu, \sigma) \in M_{n,k}\times S_{n,k}$ we associate a state of the game, where $\mu$ determines which  $\C$ operators on the board are circled, and $\sigma$ determines the order of the times in the top row. 
\begin{equation} \label{board}
\left[ 
\begin{array}{ccccc}  
t_{\sigma^{-1}(k+2)}     & t_{\sigma^{-1}(k+4)} & ......&  t_{\sigma^{-1}(k+2n)}  & \vspace{10pt} \\
\C_{1,k+2}    & \C_{1,k+4}          & ......&  \C_{1, k+2n} & \text{row } 1\\
\C_{2,k+2} & \C_{2,k+4}  & ......&   \C_{2, k+2n} & \text{row } 2 \\
...         & ...& ......& ... & ... \\
...        & \Circled{\C_{\mu(k+4), k+4}} & .....& ... &  ... \\
...         & ...& ......& ... & ... \\
\Circled{\C_{\mu(k+2), k+2}}        & ... & ......& ... &  ... \\
...         & ...& ......& ... & ... \\
    \C_{k, k+2}                   & \C_{k,k+4}          & ......& ... & ... \\
    0                                 & \C_{k+1,k+4}       & ......&  \Circled{\C_{\mu(k+2n),k+2n}}& ... \\
    0                                 & \C_{k+2,k+4}       & ......& ... & ... \\
    ...                                & 0                        & ......& ... & ...\\ 
    ...                                & ...                       & ......& ...& ...\\
  0                                & 0                       & ......&  \C_{k+2n-2,k+2n}& \text{row } k+2n-3\\
    0                                 & 0                        & ......&  \C_{k+2n-2,k+2n} & \text{row } k+2n-2\vspace{8pt}\\
    \text{col } k+2	     &   \text{col } k+4  & ......&  \text{col } k+2n &
\end{array} 
\right] \,.
\end{equation}

The strategy is to define an "acceptable move" on the board which will allow us to move the circles in such a way that the value of the iterated Duhamel integral \eqref{single iterated Duhamel}  is invariant. Ultimately, this will  enable us to define an equivalence relation between integrals of the type \eqref{single iterated Duhamel}, and the sum over all integrals in the same equivalence class will be estimated by a single time integral (see Proposition \ref{sum within one equivalence class}), while the number of equivalence classes will be exponential (see Proposition \ref{number of echelon forms}), thus resolving the issue of the factorial number of terms.

Inspired by \cite{chpa11} we define an acceptable move as follows.
\begin{definition}[Acceptable move] \label{def-acceptable move}
Let $k,n \in \N$ with $n\ge 2$, and let $M_{k,n}$ and $S_{n,k}$ be defined as in \eqref{set M} and \eqref{S}. Suppose $(\mu,\sigma) \in M_{n,k} \times S_{n,k}$ is a state of the game for which one has $\mu(j+2) < \mu(j)$ for some $j \in \{ k+2, k+4, \dots, k+2n-2\}$. An acceptable move transforms $(\mu, \sigma)$ to $(\mu', \sigma')$, where
\begin{equation}\label{acceptable move rule}
\begin{aligned}
    &\mu' = (j-1, j+1) \circ (j,j+2) \circ \mu \circ (j, j+2),\\
    &\sigma'=(j,j+2) \circ \sigma.
\end{aligned}
\end{equation}
\end{definition}
Note that under an acceptable move, due to $\mu(j+2) <\mu(j) <j-1$ and $\mu(j-1) < j-2$, we have
\begin{equation}\label{mu' calculation}
\begin{aligned}
\mu'(\ell) = 
\begin{cases}
  (j-1, j+1) \circ (j,j+2) \circ \mu(\ell) & \text{for \,} \ell  \in \{k+2, k+4, \dots k+2n \} \setminus \{j, j+2\}, \\
\mu(j+2) & \text{for \,} \ell = j,\\
\mu(j) & \text{for \,} \ell = j+2,
\end{cases}
\end{aligned}
\end{equation}
and 
\begin{equation}\label{sigma' calculation}
\begin{aligned}
\sigma'^{-1}(\ell) = 
\begin{cases}
 \sigma^{-1}(\ell) & \text{for \,} \ell  \in \{k+2, k+4, \dots k+2n \} \setminus \{j, j+2\}, \\
\sigma^{-1}(j+2) & \text{for \,} \ell = j,\\
\sigma^{-1}(j) & \text{for \,} \ell = j+2.
\end{cases}
\end{aligned}
\end{equation}

Therefore, the effect of an acceptable move to the board is:
\begin{itemize}
\item it exchanges positions of times in column $j$ and column $j+2$ (due to $\sigma'$), and
\item it exchanges positions of  circles in column $j$ and  column $j+2$ (due to $\mu'$), and
\item  it exchanges  positions of  circles in row $j $ and row $j+2$ if such rows exist (due to $\mu'$) (if one of those rows does not exist, no changes are made at the level of rows), and
\item  it exchanges  positions of  circles in row $j-1 $ and row $j+1$ if such rows exist (due to $\mu'$) (if one of those rows does not exist, no changes are made at the level of rows).
\end{itemize}

Before we show that an acceptable move doesn't change the value of the integral $\I_{n,k}(\mu, \sigma)$, let us introduce the following operator.
\begin{definition}[Operator $S_{j,j+2}$]
 We define  $S_{j, j+2}$ to be an operator that exchanges $x$ variables in positions $j-1,j$ with $x$ variables in positions  $j+1, j+2$, and exchanges $v$ variables in positions $j-1,j$ with $v$ variables in positions  $j+1,j+2$. In other words, for $\ell > j+1$, we define
\begin{align}\label{S operator new}
    & \Big[S_{j, j+2} f^{(\ell)}\Big](X_\ell, V_\ell) :=  
    f^{(\ell)}(x_1, \dots, x_{j-2},  {\color{blue} x_{j+1}, x_{j+2}} {\color{red}, x_{j-1} ,x_j }, x_{j+3}, \dots, x_\ell; \nonumber \\
    & \hspace{4.5cm}  v_1, \dots, v_{j-2},{\color{blue} v_{j+1}, v_{j+2}}, {\color{red} v_{j-1}, v_j }, v_{j+3}, \dots, v_\ell),
\end{align}
\end{definition}

Now we are ready to state and prove the invariance of integrals $\I_{n,k}(\mu, \sigma)$ under acceptable moves.
\begin{proposition}[Acceptable move invariance]
 Suppose $(\mu, \sigma)$ and $(\mu',\sigma')$ are as in Definition \ref{def-acceptable move}.
Then $(\mu',\sigma') \in M_{n,k}\times S_{n,k}$ and 
\begin{align*}
    \I_{n,k}(\mu, \sigma) = \I_{n,k}(\mu', \sigma').
\end{align*}
\end{proposition}
\begin{proof} Suppose $j$ is as in Definition \ref{def-acceptable move}, that is, suppose $j \in \{ k+2, k+4, \dots, k+2n-2\}$ is such that $\mu(j+2) < \mu(j)$. We first write what $\I_{n,k} (\mu', \sigma') f^{(k+2n)}$ is:
\begin{align}\label{I'}
\I_{n,k}& (\mu', \sigma') f^{(k+2n)}
     = \int_{t_k \ge t_{k+2} \ge t_{k+4} \ge \dots \ge t_{k+2n} \ge 0} \, 
     J_{n,k} (\sigma'^{-1}(\t_{k+2n}); \mu')
     \, dt_{k+2n} \dots dt_{k+4} dt_{k+2} \nonumber\\
     & = \int_{t_k \geq t_{k+2} \geq t_{k+4} \geq \dots \geq t_{k+2n}\geq 0} 
  		 T_k^{- t_{\sigma'^{-1}(k+2)}} \C_{\mu'(k+2),k+2} \, T_{k+2}^{t_{\sigma'^{-1}(k+2)}- t_{\sigma'^{-1}(k+4)}} \C_{\mu'(k+4),k+4} \nonumber  \\
     & \qquad \qquad \dots T_{j-2}^{t_{\sigma'^{-1}(j-2)} - t_{\sigma'^{-1}(j)}}
    \C_{ \mu'(j) ,j } 
    \,\, T_{j}^{t_{\sigma'^{-1}(j)}- t_{\sigma'^{-1}(j+2)}} 
    \C_{ \mu'(j+2),j+2}  \nonumber 
    \,\, T_{j+2}^{t_{\sigma'^{-1}(j+2)} -t_{\sigma'^{-1}(j+4)}} 
  \dots \nonumber \\
  & \qquad \qquad \cdots\,  T^{t_{\sigma'^{-1}(k+2n-2)} - t_{\sigma'^{-1}(k+2n)}}_{k+2n-2} \C_{\mu'(k+2n), k+2n}
  \, f^{(k+2n)}(t_{\sigma'^{-1}(k+2n)}) dt_{k+2n}  \dots dt_{k+2}.
\end{align}

According to properties \eqref{mu' calculation} and \eqref{sigma' calculation}, for the operators appearing before  $T_{j-2}$ in \eqref{I'}, each $\mu'$ and $\sigma'^{-1}$ can be replaced by $\mu$ and $\sigma^{-1}$, respectively. For operators appearing after $T_{j+2}$ in \eqref{I'} one can drop primes from $\sigma'^{-1}$ and turn $\mu'$  to $(j-1, j+1) \circ (j,j+2)\circ \mu$. Finally, for the operators appearing between and including $T_{j-2}$ and $T_{j+2}$ in \eqref{I'}, 
the evaluation of $\sigma'^{-1}$ and $\mu'$ at $j$ (resp. $j+2$) turns into an evaluation of $\sigma^{-1}$ and $\mu$ at $j+2$ (resp. $j$). For all other $\sigma'^{-1}$, one can drop the prime. Therefore, we can rewrite $\I_{n,k} (\mu', \sigma')$ in terms of $\mu$ and $\sigma$ as follows

\begin{align*}
\I_{n,k}& (\mu', \sigma') f^{(k+2n)}
      = \int_{t_k \geq t_{k+2} \geq t_{k+4} \geq \dots \geq t_{k+2n}\geq 0} 
  		 T_k^{- t_{\sigma^{-1}(k+2)}} \C_{\mu(k+2),k+2} \, T_{k+2}^{t_{\sigma^{-1}(k+2)}- t_{\sigma^{-1}(k+4)}} \C_{\mu(k+4),k+4} \nonumber  \\
     & \qquad \dots T_{j-2}^{t_{\sigma^{-1}(j-2)} - t_{\sigma^{-1}(j+2)}}
    \C_{ \mu(j+2) ,j } 
    \,\, T_{j}^{t_{\sigma^{-1}(j+2)}- t_{\sigma^{-1}(j)}} 
    \C_{ \mu(j),j+2}  \nonumber 
    \,\, T_{j+2}^{t_{\sigma^{-1}(j)} -t_{\sigma^{-1}(j+4)}} 
  \dots \nonumber \\
  & \qquad \cdots\,  T^{t_{\sigma^{-1}(k+2n-2)} - t_{\sigma^{-1}(k+2n)}}_{k+2n-2} \C_{(j-1, j+1)\circ (j,j+2)\circ\mu(k+2n), k+2n}
  \, f^{(k+2n)}(t_{\sigma^{-1}(k+2n)}) dt_{k+2n}  \dots dt_{k+2}.
\end{align*}

In order to complete the proof of the proposition, it suffices to show two identities:
\begin{align}
    & T_{j-2}^{a - {\color{red} b}} \C_{{\color{blue}\alpha},j} T_{j}^{{\color{red}b}-{\color{green}c}} \C_{{\color{blue}\beta},j+2} T_{j+2}^{ {\color{green}c}- d} = T_{j-2}^{a - {\color{green}c}} \C_{{\color{blue}\beta}, j} T_j^{{\color{green}c}-{\color{red}b}} \C_{{\color{blue}\alpha}, j+2} T_{j+2}^{ {\color{red}b} -d} S_{j,j+2}, \label{identity 1}\\
    &   S_{j, j+2} \mathfrak{C}_{(j-1, j+1) \circ (j,j+2) \circ\mu(\ell), \ell}
   =  \C_{\mu(\ell), \ell}\, S_{j, j+2}. \label{identity 2}
\end{align}
Namely, if these two identities are true, then an application of the  identity \eqref{identity 1} with $a=t_{\sigma^{-1}(j-2)}, \, b= t_{\sigma^{-1}(j+2)}, \, c= t_{\sigma^{-1}(j)}, \, \alpha = \mu(j+2), \, \beta = \mu(j)$ one would get
\begin{align*}
\I_{n,k}& (\mu', \sigma') f^{(k+2n)}
      = \int_{t_k \geq t_{k+2} \geq t_{k+4} \geq \dots \geq t_{k+2n}\geq 0} 
  		 T_k^{- t_{\sigma^{-1}(k+2)}} \C_{\mu(k+2),k+2} \, T_{k+2}^{t_{\sigma^{-1}(k+2)}- t_{\sigma^{-1}(k+4)}} \C_{\mu(k+4),k+4} \nonumber  \\
     & \qquad \qquad \dots T_{j-2}^{t_{\sigma^{-1}(j-2)} - t_{\sigma^{-1}(j)}}
    \C_{ \mu(j) ,j } 
    \,\, T_{j}^{t_{\sigma^{-1}(j)}- t_{\sigma^{-1}(j+2)}} 
    \C_{ \mu(j+2),j+2}  \nonumber 
    \,\, T_{j+2}^{t_{\sigma^{-1}(j+2)} -t_{\sigma^{-1}(j+4)}} {\color{blue} S_{j,j+2}}
  \dots \nonumber \\
  & \qquad \qquad \cdots\,  T^{t_{\sigma^{-1}(k+2n-2)} - t_{\sigma^{-1}(k+2n)}}_{k+2n-2} \C_{(j-1, j+1) \circ (j,j+2)\circ\mu(k+2n), k+2n}
  \, f^{(k+2n)}(t_{\sigma^{-1}(k+2n)}) dt_{k+2n}  \dots dt_{k+2}.
\end{align*}
Then by using the identity \eqref{identity 2} iteratively, the fact that $S_{j,j+2}$ commutes with translation operators and  that $f^{k+2n}$ is a symmetric function, one would be able to conclude that $\I_{n,k} (\mu', \sigma') = \I_{n,k} (\mu, \sigma).$ Thus it remains to prove identities \eqref{identity 1} and \eqref{identity 2}, which will be done the next two lemmata.
\end{proof}

We first prove identity \eqref{identity 2}. 
\begin{lemma}\label{identity 2 lemma}
For $\ell > j+2$  and for any $\lambda \in \{L_0, L_1, L_2, L_3\}$, we have
\begin{align}\label{identity 2 claim}
S_{j, j+2} \C^\lambda_{(j-1, j+1) \circ (j,j+2) \circ\mu(\ell), \ell}
   =  \C^\lambda_{\mu(\ell), \ell}\, S_{j, j+2}.
\end{align}
\end{lemma}

\begin{proof}[Proof of Lemma \ref{identity 2 lemma}] We will present the proof for $\lambda=L_0$ and $\lambda = L_1$. The cases $\lambda = L_2$ and $\lambda = L_3$ are done analogously to $\lambda = L_1$. Throughout the proof, we will use notation \eqref{Sigma and Omega}:
\begin{align}
  &\Sigma_{p, q}  = v_p + v_{q-1} - v_{q} - v_{q+1},\\
  &\Omega_{p, q} = |v_p|^2 + |v_{q-1}|^2 - |v_{q}|^2 - |v_{q+2}|^2.
\end{align}

\noindent $\bullet$ We first prove  \eqref{identity 2 claim} with $\lambda =L_0$ by considering three cases.

\noindent {\it Case 1:} assume that $\ell > j+2$ and  $\mu(\ell) \not \in \{j-1, j, j+1, j+2\}$.  Then  $(j-1, j+1) \circ (j,j+2) \circ \mu(\ell) = \mu(\ell)$, and so for any non-negative measurable function $f^{(\ell)}$  we have 
\begin{align*}
[&S_{j, j+2}  \C^{L_0}_{{\color{teal}(j-1, j+1) \circ (j,j+2) \circ\mu(\ell)}, \ell} f^{(\ell)}](t, X_{\ell-2}, V_{\ell-2})
= [S_{j, j+2}  \C^{L_0}_{{\color{teal}\mu(\ell)}, \ell} f^{(\ell)}](t, X_{\ell-2}, V_{\ell-2})\\
& = [ \C^{L_0}_{\mu(\ell), \ell} f^{(\ell)}](x_1, \dots,  {\color{blue} x_{j+1}, x_{j+2}}, {\color{red} x_{j-1} ,x_j }, \dots, x_{\ell-2};\,\, v_1, \dots,{\color{blue} v_{j+1},  v_{j+2}},{\color{red} v_{j-1}, v_j },  \dots, v_{\ell-2})\\
&=\int_{\R^{3d}} dv_{\ell-1} dv_{\ell} dv_{\ell+1} \,\delta(\Sigma_{\mu(\ell), \ell})\, \delta(\Omega_{\mu(\ell), \ell}) \nonumber \\
& \hspace{0.7cm} \left(f^{(\ell)}(x_1, \dots,  {\color{blue} x_{j+1},  x_{j+2}}, {\color{red}x_{j-1} ,x_j }, \dots\dots\dots, x_{\ell-2}, x_{\mu(\ell)},x_{\mu(\ell)}; \right. \\
& \hspace{1.8cm} v_1, \dots,{\color{blue} v_{j+1}, v_{j+2}},{\color{red} v_{j-1}, v_j },  \dots, \underbrace{v_{\ell-1}}_{\mu(\ell)-\text{th}}, \dots, v_{\ell-2},v_{\ell},v_{\ell+1} )\\
&=\int_{\R^{3d}} dv_{\ell-1} dv_{\ell} dv_{\ell+1} \,\delta(\Sigma_{\mu(\ell), \ell})\, \delta(\Omega_{\mu(\ell), \ell}) 
[S_{j,j+2}f^{(\ell)}](X_{\ell-2}, x_{\mu(\ell)},x_{\mu(\ell)}; \,\,V_{\ell-2}^{\mu(\ell),v_{\ell-1}},v_{\ell},v_{\ell+1} )\\
& = [\C^{L_0}_{\mu(\ell),\ell} S_{j,j+2} f^{(\ell)}](t,X_{\ell-2}, V_{\ell-2}).
\end{align*}

\noindent {\it Case 2:} assume that $\ell > j+2$ and  $\mu(\ell) \in \{j, j+2\}$. Without loss of generality, $\mu(\ell) = j$.  Then  $(j-1, j+1) \circ (j,j+2) \circ \mu(\ell) = j+2$, and so for any non-negative measurable $f^{(\ell)}$  
we have 
\begin{align*}
[&S_{j, j+2}  \C^{L_0}_{{\color{teal}(j-1, j+1) \circ (j,j+2) \circ\mu(\ell)}, \ell} f^{(\ell)}](t, X_{\ell-2}, V_{\ell-2})
= [S_{j, j+2}  \C^{L_0}_{{\color{teal}j+2}, \ell} f^{(\ell)}](t, X_{\ell-2}, V_{\ell-2})\\
& = [ \C^{L_0}_{j+2, \ell} f^{(\ell)}](x_1, \dots,  {\color{blue} x_{j+1}, x_{j+2}}, {\color{red} x_{j-1}, \underbrace{x_j}_{(j+2)\text{nd} }}, \dots, x_{\ell-2};\,\, v_1, \dots,{\color{blue} v_{j+1}, v_{j+2}}, {\color{red}  v_{j-1}, \underbrace{v_j}_{(j+2)\text{nd} }},  \dots, v_{\ell-2})\\
&=\int_{\R^{3d}} dv_{\ell-1} dv_{\ell} dv_{\ell+1} \,\delta(\Sigma_{j, \ell})\, \delta(\Omega_{j, \ell}) \nonumber \\
& \hspace{0.7cm} \left(f^{(\ell)}(x_1, \dots,  {\color{blue} x_{j+1}, x_{j+2}}, {\color{red}  x_{j-1} ,x_j }, \dots\dots, x_{\ell-2}, x_{j},x_{j}; \right. \\
& \hspace{1.8cm} v_1, \dots,{\color{blue} v_{j+1}, v_{j+2}},{\color{red} v_{j-1}}, {\color{teal} \underbrace{v_{\ell-1}}_{(j+2)\text{nd} }},  \dots, v_{\ell-2},v_{\ell},v_{\ell+1} )\\
& =\int_{\R^{3d}} dv_{\ell-1} dv_{\ell} dv_{\ell+1} \,\delta(\Sigma_{j, \ell})\, \delta(\Omega_{j, \ell}) 
[S_{j,j+2} f^{(\ell)}](X_{\ell-2}, x_{j},x_{j};  \,\, V_{\ell-2}^{j, v_{\ell-1}},v_{\ell},v_{\ell+1} )\\
& = [\C^{L_0}_{j, \ell} S_{j,j+2} f^{(\ell)}](t, X_{\ell-2}, V_{\ell-2})
= [\C^{L_0}_{\mu(\ell), \ell} S_{j,j+2} f^{(\ell)}](t, X_{\ell-2}, V_{\ell-2}).
\end{align*}

\noindent {\it Case 3:} assume that $\ell > j+2$ and  $\mu(\ell) \in \{j-1, j+1\}$. Without loss of generality,  $\mu(\ell) = j-1$.  Then  $(j-1, j+1)\circ (j,j+2) \circ \mu(\ell) = j+1$, and so for any non-negative measurable $f^{(\ell)}$  we have 

\begin{align*}
[&S_{j, j+2}  \C^{L_0}_{{\color{teal}(j-1, j+1)\circ (j,j+2) \circ\mu(\ell)}, \ell} f^{(\ell)}](t, X_{\ell-2}, V_{\ell-2})
= [S_{j, j+2}  \C^{L_0}_{{\color{teal}j+1}, \ell} f^{(\ell)}](t, X_{\ell-2}, V_{\ell-2})\\
& = [ \C^{L_0}_{j+1, \ell} f^{(\ell)}](x_1, \dots,  {\color{blue} x_{j+1}, x_{j+2}}, {\color{red} \underbrace{x_{j-1}}_{(j+1)\text{st}} ,x_j }, \dots, x_{\ell-2};\,\, v_1, \dots, {\color{blue} v_{j+1},  v_{j+2}}, {\color{red} \underbrace{v_{j-1}}_{(j+1)\text{st}}, v_j} ,  \dots, v_{\ell-2})\\
&=\int_{\R^{3d}} dv_{\ell-1} dv_{\ell} dv_{\ell+1} \,\delta(\Sigma_{j-1, \ell})\, \delta(\Omega_{j-1, \ell}) \nonumber \\
& \hspace{0.7cm} \left(f^{(\ell)}(x_1, \dots,  {\color{blue} x_{j+1}, x_{j+2}}, {\color{red}  x_{j-1} ,x_j }, \dots, x_{\ell-2}, x_{j-1},x_{j-1}; \right. \\
& \hspace{1.8cm} v_1, \dots,{\color{blue} v_{j+1}, v_{j+2}},{\color{teal} \underbrace{v_{\ell-1}}_{(j+1)\text{st}}}, {\color{red} v_{j}},  \dots, v_{\ell-2},v_{\ell},v_{\ell+1} )\\
& =\int_{\R^{3d}} dv_{\ell-1} dv_{\ell} dv_{\ell+1} \,\delta(\Sigma_{j-1, \ell})\, \delta(\Omega_{j-1, \ell}) 
[S_{j,j+2} f^{(\ell)}](X_{\ell-2}, x_{j-1},x_{j-1};  \,\, V_{\ell-2}^{j-1, v_{\ell-1}},v_{\ell},v_{\ell+1} )\\
& = [\C^{L_0}_{j-1, \ell} S_{j,j+2} f^{(\ell)}](t, X_{\ell-2}, V_{\ell-2})
= [\C^{L_0}_{\mu(\ell), \ell} S_{j,j+2} f^{(\ell)}](t, X_{\ell-2}, V_{\ell-2}).
\end{align*}

\noindent $\bullet$ We next prove \eqref{identity 2} with $\lambda = L_1$ by again considering three cases.

\noindent {\it Case 1:} assume that $\ell > j+2$ and  $\mu(\ell) \not \in \{j-1, j, j+1, j+2\}$.  Then,  $(j-1, j+1)\circ (j,j+2) \circ \mu(\ell) = \mu(\ell)$,  and so for any non-negative measurable function $f^{(\ell)}$  we have 
\begin{align*}
[&S_{j, j+2}  \C^{L_1}_{{\color{teal}(j-1,j+1)\circ (j,j+2) \circ\mu(\ell)}, \ell} f^{(\ell)}](t, X_{\ell-2}, V_{\ell-2})
= [S_{j, j+2}  \C^{L_1}_{{\color{teal}\mu(\ell)}, \ell} f^{(\ell)}](t, X_{\ell-2}, V_{\ell-2})\\
& = [ \C^{L_1}_{\mu(\ell), \ell} f^{(\ell)}](x_1, \dots,  {\color{blue} x_{j+1}, x_{j+2}}, {\color{red} x_{j-1} ,x_j }, \dots, x_{\ell-2};\,\, v_1, \dots,{\color{blue} v_{j+1}, v_{j+2}}, {\color{red}  v_{j-1}, v_j },  \dots, v_{\ell-2})\\
&=\int_{\R^{3d}} dv_{\ell-1} dv_{\ell} dv_{\ell+1} \,\delta(\Sigma_{\mu(\ell), \ell})\, \delta(\Omega_{\mu(\ell), \ell}) 
f^{(\ell)}(x_1, \dots,  {\color{blue} x_{j+1}, x_{j+2}}, {\color{red} x_{j-1} ,x_j }, \dots, x_{\ell-2}, x_{\mu(\ell)},x_{\mu(\ell)};  \\
& \hspace{7cm} v_1, \dots,{\color{blue} v_{j+1}, v_{j+2}}, {\color{red} v_{j-1}, v_j },   \dots, v_{\ell-2},v_{\ell},v_{\ell+1} )  \nonumber \\
& = \int_{\R^{3d}} dv_{\ell-1} dv_{\ell} dv_{\ell+1} \,\delta(\Sigma_{\mu(\ell), \ell})\, \delta(\Omega_{\mu(\ell), \ell})  [S_{j,j+2} f^{(\ell)}](X_{\ell-2}, x_{\mu(\ell)},x_{\mu(\ell)}; \,\,V_{\ell-2},v_{\ell},v_{\ell+1} )\\
& = [\C^{L_1}_{\mu(\ell),\ell} S_{j,j+2} f^{(\ell)}](t,X_{\ell-2}, V_{\ell-2}).
\end{align*}

\noindent {\it Case 2:} assume that $\ell > j+2$ and  $\mu(\ell) \in \{j, j+2\}$.  Without loss of generality, assume that $\mu(\ell) = j$. Then,  $(j-1, j+1)\circ (j,j+2) \circ \mu(\ell) = j+2$,  and so for any non-negative measurable function $f^{(\ell)}$  we have 
\begin{align*}
[&S_{j, j+2}  \C^{L_1}_{{\color{teal}(j-1, j+1) \circ (j,j+2) \circ\mu(\ell)}, \ell} f^{(\ell)}](t, X_{\ell-2}, V_{\ell-2})
= [S_{j, j+2}  \C^{L_1}_{{\color{teal}j+2}, \ell} f^{(\ell)}](t, X_{\ell-2}, V_{\ell-2})\\
& = [ \C^{L_1}_{j+2, \ell} f^{(\ell)}](x_1, \dots,  {\color{blue} x_{j+1}, x_{j+2}}, {\color{red} x_{j-1} ,\underbrace{x_j}_{(j+2)\text{nd}} }, \dots, x_{\ell-2};\,\, v_1, \dots,{\color{blue} v_{j+1}, v_{j+2}}, {\color{red} v_{j-1}, \underbrace{v_j}_{(j+2)\text{nd}} },  \dots, v_{\ell-2})\\
&=\int_{\R^{3d}} dv_{\ell-1} dv_{\ell} dv_{\ell+1} \,\delta(\Sigma_{j, \ell})\, \delta(\Omega_{j, \ell}) 
 f^{(\ell)}(x_1, \dots,  {\color{blue} x_{j+1}, x_{j+2}}, {\color{red} x_{j-1} ,x_j }, \dots, x_{\ell-2}, x_{j},x_{j};  \\
& \hspace{6.3cm} v_1, \dots,{\color{blue} v_{j+1}, v_{j+2}}, {\color{red} v_{j-1}, v_j },   \dots, v_{\ell-2},v_{\ell},v_{\ell+1} )  \nonumber \\
& =\int_{\R^{3d}} dv_{\ell-1} dv_{\ell} dv_{\ell+1} \,\delta(\Sigma_{j, \ell})\, \delta(\Omega_{j, \ell})[S_{j,j+2} f^{(\ell)}](X_{\ell-2}, x_{j},x_{j}; \,\,V_{\ell-2},v_{\ell},v_{\ell+1} ) \\
& = [\C^{L_1}_{j, \ell} S_{j,j+2} f^{(\ell)}](t, X_{\ell-2}, V_{\ell-2})
= [\C^{L_1}_{\mu(\ell), \ell} S_{j,j+2} f^{(\ell)}](t, X_{\ell-2}, V_{\ell-2}).
\end{align*}

\noindent {\it Case 3:} assume that $\ell > j+2$ and  $\mu(\ell) \in \{j-1, j+1\}$.  Without loss of generality, assume that $\mu(\ell) = j-1$. Then,  $(j-1, j+1)\circ (j,j+2) \circ \mu(\ell) = j+1$,  and so for any non-negative measurable function $f^{(\ell)}$  we have 
\begin{align*}
[&S_{j, j+2}  \C^{L_1}_{{\color{teal}(j-1, j+1) \circ (j,j+2) \circ\mu(\ell)}, \ell} f^{(\ell)}](t, X_{\ell-2}, V_{\ell-2})
= [S_{j, j+2}  \C^{L_1}_{{\color{teal}j+1}, \ell} f^{(\ell)}](t, X_{\ell-2}, V_{\ell-2})\\
& = [ \C^{L_1}_{j+1, \ell} f^{(\ell)}](x_1, \dots,  {\color{blue} x_{j+1}, x_{j+2}}, {\color{red} \underbrace{x_{j-1}}_{(j+1)\text{st}} ,x_j }, \dots, x_{\ell-2};\,\, v_1, \dots,{\color{blue} v_{j+1}, v_{j+2}}, {\color{red} \underbrace{v_{j-1}}_{(j+1)\text{st}}, v_j },  \dots, v_{\ell-2})\\
&=\int_{\R^{3d}} dv_{\ell-1} dv_{\ell} dv_{\ell+1} \,\delta(\Sigma_{j-1, \ell})\, \delta(\Omega_{j-1, \ell}) 
 f^{(\ell)}(x_1, \dots,  {\color{blue} x_{j+1}, x_{j+2}}, {\color{red} x_{j-1} ,x_j }, \dots, x_{\ell-2}, x_{j-1},x_{j-1};  \\
& \hspace{7cm} v_1, \dots,{\color{blue} v_{j+1}, v_{j+2}}, {\color{red} v_{j-1}, v_j },   \dots, v_{\ell-2},v_{\ell},v_{\ell+1} )  \nonumber \\
& =\int_{\R^{3d}} dv_{\ell-1} dv_{\ell} dv_{\ell+1} \,\delta(\Sigma_{j-1, \ell})\, \delta(\Omega_{j-1, \ell})[S_{j,j+2} f^{(\ell)}](X_{\ell-2}, x_{j-1},x_{j-1}; \,\,V_{\ell-2},v_{\ell},v_{\ell+1} ) \\
& = [\C^{L_1}_{j-1, \ell} S_{j,j+2} f^{(\ell)}](t, X_{\ell-2}, V_{\ell-2})
= [\C^{L_1}_{\mu(\ell), \ell} S_{j,j+2} f^{(\ell)}](t, X_{\ell-2}, V_{\ell-2}).
\end{align*}
\end{proof}

Next we prove identity \eqref{identity 1}. 
\begin{lemma}\label{identity 1 lemma}
     For any  $a,b,c,d \ge 0$, any $\beta <\alpha<j$ and any $k,\ell \in \{0, 1, 2, 3\}$, we have
\begin{align}
&T_{j-2}^{a - b} \C_{\alpha,j}^{{\color{red} L_k }} T_{j}^{b-c} \C_{\beta,j+2}^{{\color{red} L_{\ell}} }  T_{j+2}^{ c- d} 
= T_{j-2}^{a - c} \C_{\beta, j}^{{\color{red}L_{\ell}}} T_j^{c-b} \C_{\alpha, j+2}^{{\color{red}L_k}} T_{j+2}^{ b -d} S_{j,j+2}
\end{align}
\end{lemma}

\begin{proof}[Proof of Lemma \ref{identity 1 lemma}]
 By applying the operator $T_{j-2}^{c-a}$ from the left and the operator $T_{j+2}^{d-b}$ from the right, the identity can be written in its equivalent form
\begin{align}
 T_{j-2}^{c - b} \C^{L_k}_{\alpha,j} T_{j}^{b-c} \C^{L_\ell}_{\beta,j+2} T_{j+2}^{ c- b} 
 = \C^{L_\ell}_{\beta, j} T_j^{c-b} \C^{L_k}_{\alpha, j+2}  S_{j,j+2}.
\end{align}
If we introduce notation
$$
\tau = b-c,
$$
it suffices to show that for any $\beta <\alpha<j$ and any $k,\ell \in \{0, 1, 2, 3\}$, we have
\begin{align}
& T_{j-2}^{-\tau} \C_{\alpha,j}^{{\color{red} L_k}} T_{j}^{\tau} \C_{\beta,j+2}^{{\color{red} L_\ell}} T_{j+2}^{ -\tau}  = \C_{\beta, j}^{{\color{red} L_\ell}} T_j^{-\tau} \C_{\alpha, j+2}^{{\color{red} L_k}}  S_{j,j+2}.
\label{identity 1  equiv}
\end{align}
The strategy of the proof is to expand the left-hand side (LHS) and the right-hand side (RHS) of these identities separately and then compare their formulas.

Note that the operator $\C_{\alpha,j}$ comes with three integrating variables $v_{j-1}, v_j$ and $v_{j+1}$, while the corresponding variables for $\C_{\beta,j+2}$ are $v_{j+1}, v_{j+2}$ and $v_{j+3}$. In order to avoid confusion due to the repeated letter $v_{j+1}$, we will add stars to the variables corresponding to $\C_{\alpha,j}$, primes for $\C_{\beta,j+2}$, sharps for $\C_{\beta,j}$
and tildes for $\C_{\alpha,j+2}$. This notation will be used only within this lemma.

\noindent $\bullet$ We first prove identity \eqref{identity 1 equiv} with $k=\ell=0$. We start by computing its left-hand side.
\begin{align*}
&\text{LHS}_{00}
= [ T_{j-2}^{-\tau} \C^{L_0}_{\alpha,j} T_{j}^{\tau} \C^{L_0}_{\beta,j+2} T_{j+2}^{ -\tau} f^{(j+2)}](t,X_{j-2}, V_{j-2}) \\
& = [ \C^{L_0}_{\alpha,j} T_{j}^{\tau} \C^{L_0}_{\beta,j+2} T_{j+2}^{ -\tau} f^{(j+2)}](t,X_{j-2} +\tau V_{j-2}, V_{j-2})\\
&= \int_{\R^{3d}} dv^*_{j-1} dv^*_{j} dv^*_{j+1} \,\delta(\Sigma^*_{\alpha, j})\, \delta(\Omega^*_{\alpha, j}) \\
& \hspace{1cm}[T_{j}^{\tau} \C^{L_0}_{\beta,j+2} T_{j+2}^{ -\tau} f^{(j+2)}](t,X_{j-2}+\tau V_{j-2}, x_\alpha +\tau v_\alpha, x_\alpha +\tau v_\alpha; \, V_{j-2}^{\alpha,v^*_{j-1}}, v^*_j, v^*_{j+1}) \\
& = \int_{\R^{3d}} dv^*_{j-1} dv^*_{j} dv^*_{j+1} \,\delta(\Sigma^*_{\alpha, j})\, \delta(\Omega^*_{\alpha, j}) \nonumber \\
& \hspace{0.5cm} [ \C^{L_0}_{\beta,j+2} T_{j+2}^{ -\tau} f^{(j+2)}](t,X_{j-2}+\tau (V_{j-2} -  V_{j-2}^{\alpha,v^*_{j-1}}), x_\alpha+\tau (v_\alpha - v^*_j), x_\alpha + \tau (v_\alpha - v^*_{j+1}); \,\,   V_{j-2}^{\alpha,v^*_{j-1}}, v^*_j, v^*_{j+1})\\
& = \int_{\R^{6d}}  dv^*_{j-1} dv^*_{j} dv^*_{j+1}  dv'_{j+1} dv'_{j+2} dv'_{j+3}  
\,\delta(\Sigma^*_{\alpha, j})\, \delta(\Omega^*_{\alpha, j})
\,\delta(\Sigma'_{\beta, j+2})\, \delta(\Omega'_{\beta, j+2}) \nonumber \\
& \hspace{1cm} [T_{j+2}^{ -\tau} f^{(j+2)}]
(t,X_{j-2}+\tau (V_{j-2} -  V_{j-2}^{\alpha,v^*_{j-1}}),  x_\alpha+\tau (v_\alpha - v^*_j),  x_\alpha+\tau (v_\alpha - v^*_{j+1}), x_\beta, x_\beta; \\
& \hspace{8cm}\,\, V_{j-2}^{\beta, v'_{j+1}; \alpha,v^*_{j-1}}, v^*_j, v^*_{j+1}, v'_{j+2}, v'_{j+3})\\
& = \int_{\R^{6d}}  dv^*_{j-1} dv^*_{j} dv^*_{j+1}  dv'_{j+1} dv'_{j+2} dv'_{j+3}  
\,\delta(\Sigma^*_{\alpha, j})\, \delta(\Omega^*_{\alpha, j})
\,\delta(\Sigma'_{\beta, j+2})\, \delta(\Omega'_{\beta, j+2}) \nonumber \\
& \hspace{1cm}  f^{(j+2)}
(t,X_{j-2}+\tau V_{j-2}^{\beta, v'_{j+1}}, x_\alpha +\tau v_\alpha, x_\alpha +\tau v_\alpha, x_\beta+\tau v'_{j+2}, x_\beta+\tau v'_{j+3}; \\
& \hspace{8cm} V_{j-2}^{\beta, v'_{j+1}; \alpha,v^*_{j-1}}, v^*_j, v^*_{j+1}, v'_{j+2}, v'_{j+3}),
\end{align*}
where in the last equality we used that
$V_{j-2} - V_{j-2}^{\alpha, v^*_{j-1}} + V_{j-2}^{\beta, v'_{j+1}; \alpha,v^*_{j-1}} =  V_{j-2}^{\beta, v'_{j+1}}.$

 We next calculate the right-hand side of \eqref{identity 1 equiv} with $k=\ell=0$.
\begin{align*}
&\text{RHS}_{00} = [\C^{L_0}_{\beta, j} T_j^{-\tau} \C^{L_0}_{\alpha, j+2}  S_{j,j+2} f^{(j+2)}](t,X_{j-2}, V_{j-2}) \\
& =  \int_{\R^{3d}} dv^\#_{j-1} dv^\#_{j} dv^\#_{j+1} \,\delta(\Sigma^\#_{\beta, j})\, \delta(\Omega^\#_{\beta, j}) [T_j^{-\tau} \C^{L_0}_{\alpha, j+2}  S_{j,j+2} f^{(j+2)}](t,X_{j-2}, x_\beta, x_\beta; \, V_{j-2}^{\beta,v^\#_{j-1}}, v^\#_j, v^\#_{j+1})\\
& = \int_{\R^{3d}} dv^\#_{j-1} dv^\#_{j} dv^\#_{j+1} \,\delta(\Sigma^\#_{\beta, j})\, \delta(\Omega^\#_{\beta, j}) \\
& \hspace{1cm}[\C^{L_0}_{\alpha, j+2}  S_{j,j+2} f^{(j+2)}](t,X_{j-2}+\tau V_{j-2}^{\beta,v^\#_{j-1}}, x_\beta+\tau v^\#_{j}, x_\beta+\tau v^\#_{j+1}; \, V_{j-2}^{\beta,v^\#_{j-1}}, v^\#_j, v^\#_{j+1})\\
& = \int_{\R^{6d}} dv^\#_{j-1} dv^\#_{j} dv^\#_{j+1} d\tild{v}_{j+1} d\tild{v}_{j+2} d\tild{v}_{j+3} \,\delta(\Sigma^\#_{\beta, j})\, \delta(\Omega^\#_{\beta, j}) \delta(\tild{\Sigma}_{\alpha, j+2})\, \delta(\tild{\Omega}_{\alpha, j+2}) \\
&\hspace{1cm}  [S_{j,j+2} f^{(j+2)}](t,X_{j-2}+\tau V_{j-2}^{\beta,v^\#_{j-1}}, x_\beta+\tau v^\#_{j}, x_\beta+\tau v^\#_{j+1}, x_\alpha+\tau v_\alpha, x_\alpha + \tau v_\alpha;\\
& \hspace{7cm} \, V_{j-2}^{\beta,v^\#_{j-1}; \, \alpha, \tild{v}_{j+1}}, v^\#_j, v^\#_{j+1}, \tild{v}_{j+2}, \tild{v}_{j+3})\\
& = \int_{\R^{6d}} dv^\#_{j-1} dv^\#_{j} dv^\#_{j+1} d\tild{v}_{j+1} d\tild{v}_{j+2} d\tild{v}_{j+3} \,\delta(\Sigma^\#_{\beta, j})\, \delta(\Omega^\#_{\beta, j}) \delta(\tild{\Sigma}_{\alpha, j+2})\, \delta(\tild{\Omega}_{\alpha, j+2}) \\
&\hspace{1cm}   f^{(j+2)}(t,X_{j-2}+\tau V_{j-2}^{\beta,v^\#_{j-1}},  x_\alpha +\tau v_\alpha, x_\alpha + \tau v_\alpha, x_\beta+\tau v^\#_{j}, x_\beta+\tau v^\#_{j+1};\\
& \hspace{8cm} V_{j-2}^{\beta,v^\#_{j-1}; \, \alpha, \tild{v}_{j+1}}, \tild{v}_{j+2}, \tild{v}_{j+3}, v^\#_j, v^\#_{j+1}).
\end{align*}
By applying the change of variables:
\begin{equation}\label{change}
\begin{aligned}
\begin{cases}
v_{j-1}^\# &\mapsto v_{j+1}'\\
v_{j}^\# &\mapsto v_{j+2}'\\
v_{j+1}^\# &\mapsto v_{j+3}'
\end{cases}
\qquad \text{and } \qquad
\begin{cases}
\tild{v}_{j+1} &\mapsto v_{j-1}^*\\
\tild{v}_{j+2} &\mapsto v_{j}^*\\
\tild{v}_{j+3} &\mapsto v_{j+1}^*
\end{cases}
\end{aligned}
\end{equation}
one can see that $\text{RHS}_{00}=\text{LHS}_{00}$, which completes the proof of \eqref{identity 1 equiv} with $k=\ell=0$.

\noindent $\bullet$ Next we prove identity \eqref{identity 1 equiv}  for $k=0$ and $\ell=1$. When $\ell=2$ or $\ell=3$, the proof can be done analogously.
We start by expanding the left-hand side:
\begin{align*}
&\text{LHS}_{01} = [T_{j-2}^{-\tau} \C^{L_0}_{\alpha,j} T_{j}^{\tau} \C^{L_1}_{\beta,j+2} T_{j+2}^{ -\tau} f^{(j+2)}](t, X_{j-2}, V_{j-2})\\
& = [ \C^{L_0}_{\alpha,j} T_{j}^{\tau} \C^{L_1}_{\beta,j+2} T_{j+2}^{ -\tau} f^{(j+2)}](t,X_{j-2} +\tau V_{j-2}, V_{j-2})\\
&= \int_{\R^{3d}} dv^*_{j-1} dv^*_{j} dv^*_{j+1} \,\delta(\Sigma^*_{\alpha, j})\, \delta(\Omega^*_{\alpha, j}) \\
& \hspace{1cm}[T_{j}^{\tau} \C^{L_1}_{\beta,j+2} T_{j+2}^{ -\tau} f^{(j+2)}](t,X_{j-2}+\tau V_{j-2}, x_\alpha +\tau v_\alpha, x_\alpha + \tau v_\alpha; \, V_{j-2}^{\alpha,v^*_{j-1}}, v^*_j, v^*_{j+1}) \\
& = \int_{\R^{3d}} dv^*_{j-1} dv^*_{j} dv^*_{j+1} \,\delta(\Sigma^*_{\alpha, j})\, \delta(\Omega^*_{\alpha, j}) \nonumber \\
& \hspace{0.5cm} [ \C^{L_1}_{\beta,j+2} T_{j+2}^{ -\tau} f^{(j+2)}](t,X_{j-2}+\tau (V_{j-2} -  V_{j-2}^{\alpha,v^*_{j-1}}), x_\alpha+\tau (v_\alpha-v^*_j), x_\alpha+ \tau (v_\alpha - v^*_{j+1}); \,\,   V_{j-2}^{\alpha,v^*_{j-1}}, v^*_j, v^*_{j+1})\\
& = \int_{\R^{6d}}  dv^*_{j-1} dv^*_{j} dv^*_{j+1}  dv'_{j+1} dv'_{j+2} dv'_{j+3}  
\,\delta(\Sigma^*_{\alpha, j})\, \delta(\Omega^*_{\alpha, j})
\,\delta(\Sigma'_{\beta, j+2})\, \delta(\Omega'_{\beta, j+2}) \nonumber \\
& \hspace{1cm} [T_{j+2}^{ -\tau} f^{(j+2)}]
(t,X_{j-2}+\tau (V_{j-2} -  V_{j-2}^{\alpha,v^*_{j-1}}), x_\alpha+ \tau (v_\alpha - v^*_{j}), x_\alpha+ \tau (v_\alpha - v^*_{j+1}), x_\beta, x_\beta; \\
& \hspace{9.5cm}\,\, V_{j-2}^{ \alpha,v^*_{j-1}}, v^*_j, v^*_{j+1}, v'_{j+2}, v'_{j+3})\\
& = \int_{\R^{6d}}  dv^*_{j-1} dv^*_{j} dv^*_{j+1}  dv'_{j+1} dv'_{j+2} dv'_{j+3}  
\,\delta(\Sigma^*_{\alpha, j})\, \delta(\Omega^*_{\alpha, j})
\,\delta(\Sigma'_{\beta, j+2})\, \delta(\Omega'_{\beta, j+2}) \nonumber \\
& \hspace{.5cm} f^{(j+2)}
(t,X_{j-2}+\tau V_{j-2}, x_\alpha+ \tau v_\alpha, x_\alpha+ \tau v_\alpha, x_\beta + \tau v'_{j+2}, x_\beta +\tau v'_{j+3};  V_{j-2}^{ \alpha,v^*_{j-1}}, v^*_j, v^*_{j+1}, v'_{j+2}, v'_{j+3}).
\end{align*}

On the other hand, the right-hand side of \eqref{identity 1 equiv}  for $k=0$ and $\ell=1$ can be expanded as follows
\begin{align*}
&\text{RHS}_{01}
= [\C^{L_1}_{\beta, j} T_j^{-\tau} \C^{L_0}_{\alpha, j+2}  S_{j,j+2} f^{(j+2)}](t,X_{j-2}, V_{j-2}) \\
& =  \int_{\R^{3d}} dv^\#_{j-1} dv^\#_{j} dv^\#_{j+1} \,\delta(\Sigma^\#_{\beta, j})\, \delta(\Omega^\#_{\beta, j}) [T_j^{-\tau} \C^{L_0}_{\alpha, j+2}  S_{j,j+2} f^{(j+2)}](t,X_{j-2}, x_\beta, x_\beta; \, V_{j-2}, v^\#_j, v^\#_{j+1})\\
& =  \int_{\R^{3d}} dv^\#_{j-1} dv^\#_{j} dv^\#_{j+1} \,\delta(\Sigma^\#_{\beta, j})\, \delta(\Omega^\#_{\beta, j}) [\C^{L_0}_{\alpha, j+2}  S_{j,j+2} f^{(j+2)}](t,X_{j-2}+\tau V_{j-2}, x_\beta +\tau v^\#_j, x_\beta + \tau v^\#_{j+1}; \, V_{j-2}, v^\#_j, v^\#_{j+1})\\
& = \int_{\R^{6d}} dv^\#_{j-1} dv^\#_{j} dv^\#_{j+1} d\tild{v}_{j+1} d\tild{v}_{j+2} d\tild{v}_{j+3} \,\delta(\Sigma^\#_{\beta, j})\, \delta(\Omega^\#_{\beta, j}) \delta(\tild{\Sigma}_{\alpha, j+2})\, \delta(\tild{\Omega}_{\alpha, j+2}) \\
&\hspace{1cm}  [S_{j,j+2} f^{(j+2)}](t,X_{j-2}+\tau V_{j-2}, x_\beta+\tau v^\#_{j}, x_\beta+\tau v^\#_{j+1}, x_\alpha+\tau v_\alpha, x_\alpha + \tau v_\alpha; V_{j-2}^{\alpha, \tild{v}_{j+1}}, v^\#_j, v^\#_{j+1}, \tild{v}_{j+2}, \tild{v}_{j+3})\\
& = \int_{\R^{6d}} dv^\#_{j-1} dv^\#_{j} dv^\#_{j+1} d\tild{v}_{j+1} d\tild{v}_{j+2} d\tild{v}_{j+3} \,\delta(\Sigma^\#_{\beta, j})\, \delta(\Omega^\#_{\beta, j}) \delta(\tild{\Sigma}_{\alpha, j+2})\, \delta(\tild{\Omega}_{\alpha, j+2}) \\
&\hspace{1cm}   f^{(j+2)}(t,X_{j-2}+\tau V_{j-2}, x_\alpha+\tau v_\alpha, x_\alpha + \tau v_\alpha, x_\beta+\tau v^\#_{j}, x_\beta+\tau v^\#_{j+1}; V_{j-2}^{\alpha, \tild{v}_{j+1}}, \tild{v}_{j+2}, \tild{v}_{j+3}, v^\#_j, v^\#_{j+1}).
\end{align*}
Under the same change of variables as in \eqref{change}, we see that $\text{RHS}_{01} = \text{LHS}_{01}$, which completes the proof of the identity \eqref{identity 1 equiv}  for $k=0$ and $\ell=1$.

\noindent $\bullet$ Finally we prove identity \eqref{identity 1 equiv}  for $k=1$ and $\ell=2$. Any other combination of $k,\ell \in \{1,2,3\}$ can be proved analogously.
We start by expanding the left-hand side:
\begin{align*}
&\text{LHS}_{12} = [T_{j-2}^{-\tau} \C^{L_1}_{\alpha,j} T_{j}^{\tau} \C^{L_2}_{\beta,j+2} T_{j+2}^{ -\tau} f^{(j+2)}](t, X_{j-2}, V_{j-2})\\
& = [ \C^{L_1}_{\alpha,j} T_{j}^{\tau} \C^{L_2}_{\beta,j+2} T_{j+2}^{ -\tau} f^{(j+2)}](t,X_{j-2} +\tau V_{j-2}, V_{j-2})\\
&= \int_{\R^{3d}} dv^*_{j-1} dv^*_{j} dv^*_{j+1} \,\delta(\Sigma^*_{\alpha, j})\, \delta(\Omega^*_{\alpha, j}) \\
& \hspace{1cm}[T_{j}^{\tau} \C^{L_2}_{\beta,j+2} T_{j+2}^{ -\tau} f^{(j+2)}](t,X_{j-2}+\tau V_{j-2}, x_\alpha +\tau v_\alpha, x_\alpha + \tau v_\alpha; \, V_{j-2}, v^*_j, v^*_{j+1}) \\
& = \int_{\R^{3d}} dv^*_{j-1} dv^*_{j} dv^*_{j+1} \,\delta(\Sigma^*_{\alpha, j})\, \delta(\Omega^*_{\alpha, j}) \nonumber \\
& \hspace{0.5cm} [ \C^{L_2}_{\beta,j+2} T_{j+2}^{ -\tau} f^{(j+2)}](t,X_{j-2}, x_\alpha+\tau (v_\alpha-v^*_j), x_\alpha+ \tau (v_\alpha - v^*_{j+1}); \,\,   V_{j-2}, v^*_j, v^*_{j+1})\\
& = \int_{\R^{6d}}  dv^*_{j-1} dv^*_{j} dv^*_{j+1}  dv'_{j+1} dv'_{j+2} dv'_{j+3}  
\,\delta(\Sigma^*_{\alpha, j})\, \delta(\Omega^*_{\alpha, j})
\,\delta(\Sigma'_{\beta, j+2})\, \delta(\Omega'_{\beta, j+2}) \nonumber \\
& \hspace{1cm} [T_{j+2}^{ -\tau} f^{(j+2)}]
(t,X_{j-2}, x_\alpha+ \tau (v_\alpha - v^*_{j}), x_\alpha+ \tau (v_\alpha - v^*_{j+1}), x_\beta, x_\beta;  V_{j-2}, v^*_j, v^*_{j+1}, v'_{j+1}, v'_{j+3})\\
& = \int_{\R^{6d}}  dv^*_{j-1} dv^*_{j} dv^*_{j+1}  dv'_{j+1} dv'_{j+2} dv'_{j+3}  
\,\delta(\Sigma^*_{\alpha, j})\, \delta(\Omega^*_{\alpha, j})
\,\delta(\Sigma'_{\beta, j+2})\, \delta(\Omega'_{\beta, j+2}) \nonumber \\
& \hspace{.3cm}  f^{(j+2)}
(t,X_{j-2} +\tau V_{j-2}, x_\alpha+ \tau v_\alpha, x_\alpha+ \tau v_\alpha, x_\beta + \tau v'_{j+1}, x_\beta + \tau v'_{j+3};  V_{j-2}, v^*_j, v^*_{j+1}, v'_{j+1}, v'_{j+3}).
\end{align*}

On the other hand, the right-hand side of \eqref{identity 1 equiv}  for $k=1$ and $\ell=2$ can be expanded as follows
\begin{align*}
&\text{RHS}_{12}
= [\C^{L_2}_{\beta, j} T_j^{-\tau} \C^{L_1}_{\alpha, j+2}  S_{j,j+2} f^{(j+2)}](t,X_{j-2}, V_{j-2}) \\
& =  \int_{\R^{3d}} dv^\#_{j-1} dv^\#_{j} dv^\#_{j+1} \,\delta(\Sigma^\#_{\beta, j})\, \delta(\Omega^\#_{\beta, j}) [T_j^{-\tau} \C^{L_1}_{\alpha, j+2}  S_{j,j+2} f^{(j+2)}](t,X_{j-2}, x_\beta, x_\beta; \, V_{j-2}, v^\#_{j-1}, v^\#_{j+1})\\
& =  \int_{\R^{3d}} dv^\#_{j-1} dv^\#_{j} dv^\#_{j+1} \,\delta(\Sigma^\#_{\beta, j})\, \delta(\Omega^\#_{\beta, j}) [\C^{L_1}_{\alpha, j+2}  S_{j,j+2} f^{(j+2)}]\\
&\hspace{6cm}(t,X_{j-2}+\tau V_{j-2}, x_\beta +\tau v^\#_{j-1}, x_\beta + \tau v^\#_{j+1};\, V_{j-2}, v^\#_{j-1}, v^\#_{j+1})\\
& = \int_{\R^{6d}} dv^\#_{j-1} dv^\#_{j} dv^\#_{j+1} d\tild{v}_{j+1} d\tild{v}_{j+2} d\tild{v}_{j+3} \,\delta(\Sigma^\#_{\beta, j})\, \delta(\Omega^\#_{\beta, j}) \delta(\tild{\Sigma}_{\alpha, j+2})\, \delta(\tild{\Omega}_{\alpha, j+2}) [S_{j,j+2} f^{(j+2)}] \\
&\hspace{1cm}  (t,X_{j-2}+\tau V_{j-2}, x_\beta+\tau v^\#_{j-1}, x_\beta+\tau v^\#_{j+1}, x_\alpha+\tau v_\alpha, x_\alpha + \tau v_\alpha; V_{j-2}, v^\#_{j-1}, v^\#_{j+1}, \tild{v}_{j+2}, \tild{v}_{j+3})\\
& = \int_{\R^{6d}} dv^\#_{j-1} dv^\#_{j} dv^\#_{j+1} d\tild{v}_{j+1} d\tild{v}_{j+2} d\tild{v}_{j+3} \,\delta(\Sigma^\#_{\beta, j})\, \delta(\Omega^\#_{\beta, j}) \delta(\tild{\Sigma}_{\alpha, j+2})\, \delta(\tild{\Omega}_{\alpha, j+2}) \\
&\hspace{.3cm}   f^{(j+2)}(t,X_{j-2}+\tau V_{j-2}, x_\alpha+\tau v_\alpha, x_\alpha + \tau v_\alpha, x_\beta+\tau v^\#_{j-1}, x_\beta+\tau v^\#_{j+1}; V_{j-2}, \tild{v}_{j+2}, \tild{v}_{j+3}, v^\#_{j-1}, v^\#_{j+1}).
\end{align*}
Under the same change of variables as in \eqref{change}, we see that $\text{RHS}_{12} = \text{LHS}_{12}$, which completes the proof of the identity \eqref{identity 1 equiv}  for $k=1$ and $\ell=2$, and thus also the proof of this lemma.
\end{proof}

Inspired by \cite{klma08} and \cite{chpa11}, we define  a special upper echelon form as the state which does not admit any further acceptable moves.
\begin{definition}[Special Upper Echelon Form]
Let $k,n \in \N$ with $n\ge 2$, and let $M_{n,k}$  be defined as in \eqref{set M}. We say that $\mu \in M_{n,k}$ is in special upper echelon form if for every $j\in \{k+2, k+4,  \dots, k+2n\}$ we have $\mu(j) \leq \mu(j+2)$. We will denote by
\begin{align}\label{set of mu_s}
    \mathcal{M}_{n,k} ~ \text{the set of all special upper echelon forms in } M_{n,k}.
\end{align}
\end{definition}

In the same way as in \cite[Lemma 3.2]{klma08}, one can show that every state on the board can be transformed by finitely many acceptable moves into a special upper echelon form.
\begin{proposition}\label{prop-finitely many moves}
Let $k,n \in \N$ with $n\ge 2$, and let $M_{n,k}$  be defined as in \eqref{set M}. Any $\mu \in M_{n,k}$ can be changed to a special upper echelon form via a finite sequence of acceptable moves.
\end{proposition}

Next we provide an upper bound on the number of special upper echelon forms. The proof of this proposition can be done exactly in the same way as in \cite{chpa11}, since their board is of the same size as ours $(k+2n-2)\times n$.
\begin{proposition}[Number of special upper echelon forms; \protect{\cite[Lemma 7.3]{chpa11}}] 
\label{prop-equivalent classes}
Let $k,n \in \N$ with $n\ge 2$, and let $\mathcal{M}_{n,k}$  be the set of all special upper echelon forms defined  in \eqref{set of mu_s}. Then the following upper bound holds: 
\begin{align} \label{number of echelon forms}
    \# \mathcal{M}_{n,k} \leq 2^{k+3n-2}. 
\end{align}
\end{proposition}

Finally, one can show that the sum over all states that can be turned into the same special upper echelon form can be reorganized as follows. The proof of this proposition is identical to that of Theorem 7.4 in \cite{chpa11}.
\begin{proposition}[Sum over one equivalence class; \protect{\cite[Theorem 7.4]{chpa11}}]\label{sum within one equivalence class}
Let $\mu_u \in \mathcal{M}_{n,k}$ be a special upper echelon form, and write $\mu \sim \mu_u$ if $\mu$ can be reduced to $\mu_u$ in finitely many acceptable moves. Then there exists a set $D \subset [0,t]^{n}$, that depends on $\mu_u$, such that 
\begin{equation}
    \sum_{\mu \sim \mu_u} \int_{{t} \geq t_{k+2} \ge \dots \geq t_{k+2n} \geq 0} J(\underline{t}_{n,k} ; \mu) \, dt_{k+2n} \dots  d t_{k+2} 
        = \int_{D} J(\underline{t}_{n,k}; \mu_u) \, dt_{k+2n} \dots d t_{k+2},
\end{equation}
where the sum goes over all $\mu \in M_{n,k}$ such that $\mu$ can be changed to $\mu_u$ via a finite sequence of acceptable moves.
\end{proposition}

\subsubsection{Combining a priori estimates and the board game argument.} 
In this section we combine the iterated a priori estimate from Proposition \ref{prop-iterated a priori} and the board game argument described in the previous subsection to prove uniqueness of solutions to the wave kinetic hierarchy as stated in Theorem \ref{thm-uniqueness}. Since wave kinetic hierarchy in linear, it suffices to show that if $F_0=0$ and $F$ is a mild solution, then $F=0$. Here, for $F = (f^{(k)})_{k=1}^\infty$, we say $F=0$ if for each $k\in \N$ we have $f^{(k)}=0$.

\begin{proof}[Proof of Theorem \ref{thm-uniqueness}]
Recall from \eqref{iterated Duhamel} that for zero initial data, a mild solution to the wave kinetic  hierarchy can be expressed as
\begin{align}
    T^{-t}_k f^{(k)}(t)
	 &= \int_0^{t}\int_{0}^{t_{k+2}} \cdots \int_0^{t_{k+2n-2}}  dt_{k+2n} \cdots dt_{k+4} dt_{k+2} \nonumber \\
     & \qquad T_k^{- t_{k+2}} \C^{k+2} T_{k+2}^{t_{k+2}- t_{k+4}} \C^{k+4} \cdots T^{t_{k+2n-2} - t_{k+2n}}_{k+2n-2} \C^{k+2n} f^{(k+2n)}(t_{k+2n})\\
     &=  \sum_{  \mu \in M_{n,k}}  \int_0^{t}\int_{0}^{t_{k+1}} \cdots \int_0^{t_{k+n-1}} J_{n,k}(\underline{t}_n; \mu) \,\,  dt_{k+n} \cdots  dt_{k+1},
\end{align}
where the sum is taken over all the mappings in $M_{n,k}$ given by \eqref{set M}, and where $J_{n,k}(\underline{t}_n; \mu)$ is defined as in \eqref{J mu}.
By Proposition \ref{prop-finitely many moves} and Proposition \ref{sum within one equivalence class}, we can instead sum over all upper echelon forms:
\begin{align}
    T^{-t}_k f^{(k)}(t)
     &=  \sum_{  \mu_u \in \mathcal{M}_{n,k}}  \int_{D(\mu_u)} J_{n,k}(\underline{t}_n; \mu_u) \,\,  dt_{k+n} \cdots  dt_{k+1}.
\end{align} 
Since each operator $\C_{\mu_u(k+2j), k+2j}$ appearing in $J_{n,k}(\underline{t}_n; \mu_u)$ is a difference of two operators defined in \eqref{gain-loss}$:  \C^{+}_{\mu_u(k+2j), k+2j}  - \C^{-}_{\mu_u(k+2j),  k+2j}$, if we use notation $\boldsymbol{\pi} = (\pi_{k+2}, \pi_{k+4}, \dots, \pi_{k+2n}) \in \{+,-\}^n$ and  $\sign(\boldsymbol{\pi}) = \pi_{k+2} \cdot \pi_{k+4} \cdot ... \cdot \pi_{k+2n}$, we can write
\begin{align*}
    T^{-t}_k f^{(k)}(t)
     &=  \sum_{  \mu_u \in \mathcal{M}_{n,k}} \sum_{\boldsymbol{\pi} \in \{ +,-\}^n } \sign(\boldsymbol{\pi})  \int_{D(\mu_u)}  T_k^{- t_{k+2}} \C^{\pi_{k+2}}_{\mu_u(k+2),k+2} T_{k+2}^{t_{k+2}- t_{k+4}} \C^{\pi_{k+4}}_{\mu_u(k+4), k+4} \cdots \nonumber\\
& \hspace{1cm} \cdots T^{t_{k+2n-2} - t_{k+2n}}_{k+2n-2} \C^{\pi_{k+2n}}_{\mu_u(k+2n),k+2n} f^{(k+2n)}(t_{k+2n}) ~ dt_{k+2n} \cdots dt_{k+4} dt_{k+2}.
\end{align*}
Since, 
 $|\C^\pm_{j,k} g^{(k)}| \le \C^\pm_{j,k} |g^{(k)}|$  and $|T_k^{\tau}g^{(k)}| = T_k^{\tau}|g^{(k)}|$, 
 we have
\begin{align*}
    \left|T^{-t}_k f^{(k)}(t)\right|
     &\le  \sum_{  \mu_u \in \mathcal{M}_{n,k}} \sum_{\boldsymbol{\pi} \in \{ +,-\}^n } \int_{D(\mu_u)}  T_k^{- t_{k+2}} \C^{\pi_{k+2}}_{\mu_u(k+2),k+2} T_{k+2}^{t_{k+2}- t_{k+4}} \C^{\pi_{k+4}}_{\mu_u(k+4), k+4} \cdots \nonumber\\
& \hspace{1cm} \cdots T^{t_{k+2n-2} - t_{k+2n}}_{k+2n-2} \C^{\pi_{k+2n}}_{\mu_u(k+2n),k+2n} \left|f^{(k+2n)}(t_{k+2n})\right| ~ dt_{k+2n} \cdots dt_{k+4} dt_{k+2}\\
& \le   \sum_{  \mu_u \in \mathcal{M}_{n,k}} \sum_{\boldsymbol{\pi} \in \{ +,-\}^n } \int_{[0,T]^n}  T_k^{- t_{k+2}} \C^{\pi_{k+2}}_{\mu_u(k+2),k+2} T_{k+2}^{t_{k+2}- t_{k+4}} \C^{\pi_{k+4}}_{\mu_u(k+4), k+4} \cdots \nonumber\\
& \hspace{1cm} \cdots T^{t_{k+2n-2} - t_{k+2n}}_{k+2n-2} \C^{\pi_{k+2n}}_{\mu_u(k+2n),k+2n} \left|f^{(k+2n)}(t_{k+2n})\right| ~ dt_{k+2n} \cdots dt_{k+4}dt_{k+2},
\end{align*}
where in the last inequality we enlarged the domain of the time integration thanks to the fact that the integrand is non-negative.  
Multiplying both sides of the above inequality with the polynomial weights $\l\l \alpha X_k\r\r^p\l\l \beta V_k\r\r^q$, taking the supremum in $X_k,V_k$, and using the triangle inequality on the norm $\|\cdot\|_{k,p,q,\alpha,\beta}$, we obtain
\begin{align*}
    &\left\|T^{-t}_k f^{(k)}(t)\right\|_{k,p,q,\alpha,\beta}
\le   \sum_{  \mu_u \in \mathcal{M}_{n,k}} \sum_{\boldsymbol{\pi} \in \{ +,-\}^n } \left\| \int_{[0,T]^n}  T_k^{- t_{k+2}} \C^{\pi_{k+2}}_{\mu_u(k+2),k+2} T_{k+2}^{t_{k+2}- t_{k+4}} \C^{\pi_{k+4}}_{\mu_u(k+4), k+4} \cdots \nonumber \right.\\
& \hspace{2cm} \left.\cdots T^{t_{k+2n-2} - t_{k+2n}}_{k+2n-2} \C^{\pi_{k+2n}}_{\mu_u(k+2n),k+2n} \left|f^{(k+2n)}(t_{k+2n})\right| ~ dt_{k+2n} \cdots dt_{k+4}dt_{k+2}\right\|_{k,p,q,\alpha,\beta}.
\end{align*}
By Proposition \ref{prop-iterated a priori}, combined with the fact that $\#\mathcal{M}_{n,k} \le 2^{k+3n-2}$ (Proposition \eqref{prop-equivalent classes}) and $\#\{+,-\}^n = 2^n$, we have
\begin{align*}
     \left\|T^{-t}_k f^{(k)}(t)\right\|_{k,p,q,\alpha,\beta}
     &\le 2^{k+5n-2}\,C^n_{p,q,\alpha,\beta} \,
    \Big|\Big|\Big| T_{k+2n}^{-(\cdot)} \left|f^{(k+2n)} (\cdot)\right| \Big|\Big|\Big|_{k+2n, p,q,\alpha,\beta, T}\\
     &= 2^{k+5n-2}\,C^n_{p,q,\alpha,\beta} \,
    \Big|\Big|\Big| T_{k+2n}^{-(\cdot)} f^{(k+2n)} (\cdot)\Big|\Big|\Big|_{k+2n, p,q,\alpha,\beta, T}.
\end{align*}
From the definition of the norm \eqref{X_infty time norm}, we can further deduce that
\begin{align*}
\left\|T^{-t}_k f^{(k)}(t)\right\|_{k,p,q,\alpha,\beta}
& \le 2^{k+5n-2}\,C^n_{p,q,\alpha,\beta} \,
    e^{-\mu(k+2n)}\Big|\Big|\Big| \mathcal{T}^{-(\cdot)} F (\cdot)\Big|\Big|\Big|_{ p,q,\alpha,\beta, T}\\
& = \frac{(2e^{-\mu})^k}{4}\left( 32 e^{-2\mu} C_{p,q,\alpha,\beta}\right)^n
\Big|\Big|\Big| \mathcal{T}^{-(\cdot)} F (\cdot) \Big|\Big|\Big|_{ p,q,\alpha,\beta, T}.
\end{align*}
Since $\mu$ was chosen so that $e^{2\mu} > 32 C_{p,q,\alpha,\beta}$, and since $\Big|\Big|\Big| \mathcal{T}^{-(\cdot)} F (\cdot) \Big|\Big|\Big|_{ p,q,\alpha,\beta, T}<\infty$, when we let $n\to \infty$, we get that $\left\|T^{-t}_k f^{(k)}(t)\right\|_{k,p,q,\alpha,\beta} = 0$.  Since $t\in[0,T]$ was arbitrary, we obtain $T_k^{-(\cdot)} f^{(k)}(\cdot)=0$. 
Hence $f^{(k)}=0$, and thus $F=0$.
\end{proof}

\subsection{Proof of Theorem \ref{thm-gwp for kwh}}

Thanks to the assumption $e^{2\mu}>32C_{p,q,\alpha,\beta}$ in Theorem \ref{thm-gwp for kwh}, the solution constructed in Theorem \ref{thm-existence for kwh} is unique due to Theorem \ref{thm-uniqueness}.

Using representation  \eqref{kwh solution}, Fubini's theorem, and the conservation laws  \eqref{conservation of mass KWE}-\eqref{conservation of energy KWE} at the level of the wave kinetic equation, one can obtain the conservation laws \eqref{conservation of mass: KWH}-\eqref{conservation of energy: KWH} for the wave kinetic hierarchy.

Also, we prove  the stability estimate  \eqref{stability estimate hierarchy KWH} under the  assumption that the initial datum $F_0 \in  \mathcal{A}\cap \mathcal{X}_{p,q,\alpha,\beta,\mu'}^\infty$ is tensorised, i.e. $F_0=(f_0^{\otimes k})_{k=1}^\infty$. For such data, by Remark \ref{ball}, we have  $\|f_0\|_{p,q,\alpha,\beta}\leq e^{-\mu'}$ and  $F=(f^{\otimes k})_{k=1}^\infty$ is the solution to the wave kinetic hierarchy \eqref{kwh}, where $f$ is the mild solution of the wave kinetic equation with initial data $f_0$, obtained by Theorem \ref{thm-kwe is well posed}. In particular, by \eqref{bound wrt initial data equation}, we have $\|T_1^{-t}f\|_{p,q,\alpha,\beta}\leq 2\|f_0\|_{p,q,\alpha,\beta}$. Therefore, using \eqref{tensorization of transport norm} and \eqref{norm of tensors}, we obtain
\begin{align*}
e^{\mu k}\|T_k^{-t}f^{\otimes k}(t)\|_{k,p,q,\alpha,\beta}
= e^{\mu k}\|T_1^{-t}f(t)\|_{p,q,\alpha,\beta}^k
& \leq 2^ke^{\mu k} \|f_0\|_{p,q,\alpha,\beta}^k\\
&\leq e^{\mu' k} \|f_0^{\otimes k}\|_{k,p,q,\alpha,\beta}
\leq \|F_0\|_{p,q,\alpha,\beta,\mu',T}.
\end{align*}
Taking supremum over time, bound \eqref{stability estimate hierarchy} follows.

\appendix
\section{}

\subsection{Properties of tensorized functions}
\label{sec-tensorized}

Here we state results from \cite{ammipata24} regarding the relationship between the norms used in this paper and tensorized products of a given function $h:\R^{2d}\to\R$ defined by
$$h^{\otimes k}(X_k,V_k)=\prod_{i=1}^k h(x_i,v_i),\quad k\in\N.$$

\begin{remark}\label{remark on tensorization}
We note that given $k\in\N$, $h^{\otimes k}\in X_{p,q,\alpha, \beta}^k$ if and only if $h\in X_{p,q,\alpha, \beta}$. In particular, there holds
\begin{equation}\label{norm of tensors}
        \|h^{\otimes k}\|_{k,p,q,\alpha, \beta}=\|h\|_{ p,q,\alpha, \beta}^k,\quad\forall k\in\N.
    \end{equation}

\end{remark}

\begin{remark}\label{remark on tensorization of the transport} We note that the transport operator tensorizes as well. Namely for given $h:\R^{2d}\to\R$, we have
\begin{equation}\label{tensorization of the transport}
T_k^{s}h^{\otimes k}=(T_1^s h)^{\otimes k},\quad\forall \,s\in\R,\,\,\,\forall\, k\in\N.    
\end{equation}
In particular, by \eqref{norm of tensors}, we have
\begin{equation}\label{tensorization of transport norm}
 \|T_k^{s}h^{\otimes k}\|_{k,p,q,\alpha, \beta}=\|T_1^{s}h\|_{p,q,\alpha, \beta}^k,\quad\forall \,s\in\R,\,\,\,\forall\, k\in\N.    
\end{equation}
\end{remark}

\subsection{Resonant manifold properties}
\label{sec-resonant manifold}

Here we gather several properties of resonant manifolds. 
\begin{lemma}\label{lemma on resonant manifold}
Suppose $w_0,w_1, w_2, w_3 \in \R^d$ lie on the resonant manifold determined by $\Sigma = \{w_0+w_1 = w_2+w_3\}$ and $\Omega = \left\{|w_0|^2 + |w_1|^2 = |w_2|^2 + |w_3|^2\right\}.$ For any $i,j \in \{0,1,2,3\}$, define
\begin{align}
W_{i,j} = w_i - w_j,
\end{align}
and for any $v\in\R^d$, let $\hat{v}$ denote the unit vector in the direction of $v$, that is
\begin{align}
    \hat{v} = \frac{v}{|v|}.
\end{align}
Then the following identities hold
\begin{align}
 & W_{0,2} \cdot W_{0,3} =0, \label{orthogonality}\\
 & |W_{0,1}| = |W_{2,3}|, \label{magnitudes}\\
 & |W_{0,2}|^2+|W_{0,3}|^2=|W_{0,1}|^2 \label{pythagoras}
 \end{align}
Additionally, the following estimate holds
    \begin{align}\label{min estimate}
    \min\{|W_{0,2}|,|W_{0,3}|\}\geq \frac{|W_{0,1}|}{2}\sqrt{1-(\widehat{W}_{0,1}\cdot\widehat{W}_{2,3})^2}.
    \end{align}
\end{lemma}

\begin{proof}
Suppose $w_0,w_1, w_2, w_3 \in \R^d$ belong to $\Sigma$ and $\Omega$, that is, they satisfy
\begin{align}
   & w_0+w_1 = w_2+w_3, \label{momentum}\\
   & |w_0|^2 + |w_1|^2 = |w_2|^2 + |w_3|^2.\label{energy}
\end{align}
By squaring \eqref{momentum} and subtracting from it \eqref{energy}, one has that
\begin{align} \label{mixed}
     w_0\cdot w_1 = w_2 \cdot w_3.
\end{align}
Next, by multiplying the identity \eqref{momentum} by $w_0$ and combining that with \eqref{mixed}, one has 
\begin{align*}
    |w_0|^2 + w_2 + w_3 = w_0 \cdot(w_2 + w_3),
\end{align*}
which is equivalent to  \eqref{orthogonality}.

Identity \eqref{magnitudes} easily follows from \eqref{energy} and \eqref{mixed}.

By combining the orthogonality property \eqref{orthogonality} and \eqref{magnitudes}, we have
\begin{equation}\label{orthogonality equality appendix}
|W_{0,2}|^2+|W_{0,3}|^2=|W_{2,3}|^2=|W_{0,1}|^2.    
\end{equation}
Finally, we prove \eqref{min estimate}.
Momentum equation \eqref{momentum} and  \eqref{magnitudes} imply that 
\begin{align}
    w_{2}&=\frac{w_0+w_1}{2}+\frac{|W_{0,1}|}{2}\widehat{W}_{2,3}\label{w_2},  \\
    w_{3}&=\frac{w_0 + w_{1}}{2}-\frac{|W_{0,1}|}{2}\widehat{W}_{2,3}\label{w_{3}}.
\end{align}
Therefore,
\begin{align*}
|W_{0,2}|&=\frac{|W_{0,1}|}{2}|\widehat{W}_{0,1} - \widehat{W}_{2,3}|,\qquad
|W_{0,3}|=\frac{|W_{0,1}|}{2}|\widehat{W}_{0,1}+\widehat{W}_{2,3}|.
\end{align*}
and thus 
\begin{align}
|W_{0,2}|\,|W_{0,3}|
& = \frac{|W_{0,1}|^2}{4} \sqrt{|\widehat{W}_{0,1} - \widehat{W}_{2,3}|^2 \,\,\,|\widehat{W}_{0,1} + \widehat{W}_{2,3}|^2} \nonumber\\
& = \frac{|W_{0,1}|^2}{4} 
\sqrt{\left(2-2\,\widehat{W}_{0,1} \cdot \widehat{W}_{2,3}\right) \left(2+2\,\widehat{W}_{0,1} \cdot \widehat{W}_{2,3}\right)} \nonumber \\
&=\frac{|W_{0,1}|^2}{2}\sqrt{1-(\widehat{W}_{0, 1}\cdot\widehat{W}_{2,3})^2}.
\label{Carleman formula appendix}
\end{align}
Then, the estimate \eqref{min estimate} follows from  the elementary inequality $\min\{|x|,|y|\}\geq\frac{|xy|}{\sqrt{x^2+y^2}}$ and identities \eqref{orthogonality equality appendix} and \eqref{Carleman formula appendix}.

\end{proof}

\subsection{Various integral estimates }
\label{sec-integral estimates}

In this section we gather several estimates that will be used throughout the paper. We begin by two lemmas  that will be used to estimate time integrals appearing in proofs a priori estimates in Section \ref{section - a priori estimates}.

\begin{lemma}[\protect{\cite[Lemma A.1]{ammipata24}}]\label{one bracket}
For $p>1$  and $x,\eta \in \R^d$  with $\eta \ne 0$, we have
\begin{align}
    \int_{-\infty}^\infty   \l x+s\eta\r^{-p} ~ds
    \le \frac{2p}{p-1} \frac{1}{|\eta|}.
\end{align}
\end{lemma}

\begin{lemma}[\protect{\cite[Lemma 3.3]{ammipata24}}]\label{time integral}   
Let $p>1$ and  $x\in\R^d$. Consider $\xi,\eta\in\R^d$ with $\xi,\eta\neq 0 $ and $\xi\cdot \eta=0$. Then for any $t\ge 0$ there holds the bound
\begin{equation}\label{estimate on I}
    \int_0^t \l  x+s\xi\r^{-p}\l  x+s\eta\r^{-p}\,ds\leq \frac{4p}{p-1}\,\,\frac{\l  x\r^{-p}}{\min\{|\xi|,|\eta|\}}.
\end{equation} 
\end{lemma}

Our next goal is to  provide estimates (see Lemma \ref{delta convolution lemma} and Lemma \ref{appendix lemma on velocities weight}) that are used to control velocity integrals in a priori estimates in Section \ref{section - a priori estimates}. They, in turn, rely on two results - a convolution lemma from \cite{ammipata24} (Lemma \ref{convolution lemma}) and a paramertization lemma for the integration over resonant manifolds (Lemma \ref{delta lemma}). We start by recalling  the convolution lemma from \cite{ammipata24}.

\begin{lemma}[\protect{\cite[Lemma A.2.]{ammipata24}}]\label{convolution lemma}
Suppose $\delta\in (-d,0]$  and let $q>d+\delta$. Then there exists a positive constant $L_{q,\delta}$ such that 
\begin{equation}\label{convolution estimate}\int_{\R^d}|y-v|^{\delta}\l y\r^{-q}\,dy\le L_{q,\delta},\quad\forall v\in\R^d. 
\end{equation}
 One can take 
\begin{align}\label{L_{q,gamma}}
    L_{q,\delta} = \omega_{d-1} \left( \frac{1}{d} + \frac{1}{d+\delta} + \frac{2}{q -d - \delta} \right).
\end{align}
\end{lemma}

Next we prove a parametrization lemma for the integration over resonant manifolds determined by $\Sigma = v+v_1 -v_2-v_3$ and $\Omega = |v|^2 + |v_1|^2 - |v_2|^2 - |v_3|^2$.  
\begin{lemma}\label{delta lemma}
For $f:\R^{4d}\to\R$ for which the integrals below make sense, we have
$$\int_{\R^{3d}}{\delta(\Sigma)\delta(\Omega)f(v,v_1,v_2,v_3)}\,dv_1\,dv_2\,dv_3
= 2^{-d} \int_{\R^d\times\S^{d-1}}|v-v_1|^{d-2}f(v,v_1,v_2(\sigma),v_3(\sigma))\,d\sigma\,dv_1,$$
where $\Sigma = v+v_1 -v_2-v_3$, $\Omega = |v|^2 + |v_1|^2 - |v_2|^2 - |v_3|^2$, and
\begin{equation}
\begin{split}
v_2(\sigma)&=\frac{v+v_1}{2}+\frac{|v-v_1|}{2}\sigma,\quad
v_3(\sigma)=\frac{v+v_1}{2}-\frac{|v-v_1|}{2}\sigma.
\end{split}    
\end{equation}
\end{lemma}
\begin{proof}
 Let us write 
 \begin{align*}
 I(v)&=\int_{\R^{3d}}{\delta(\Sigma)\delta(\Omega)f(v,v_1,v_2,v_3)}\,dv_1\,dv_2\,dv_3\\
 &=\int_{\R^{2d}}\delta(|v_2|^2+|v+v_1-v_2|^2-|v|^2-|v_1|^2)f(v,v_1,v_2, v+v_1-v_2)\,dv_1\,dv_2.
 \end{align*} 
It is easy to verify that
 \begin{align*}
 |v_2|^2+|v+v_1-v_2|^2-|v|^2-|v_1|^2
 & = 2 \left( \left| v_2 - \frac{v+v_1}{2}\right|^2 - \left| \frac{v-v_1}{2}\right|^2 \right),
 \end{align*}
and so
 \begin{align*}
    I(v) &=\int_{\R^{2d}}
    \delta\left(2  \left| v_2 - \frac{v+v_1}{2}\right|^2 - 2\left| \frac{v-v_1}{2}\right|^2 \right)
   f(v,v_1,v_2, v+v_1-v_2)\,dv_1\,dv_2.
 \end{align*}
By letting $y = \sqrt{2}\left(v_2 - \frac{v+v_1}{2}\right)$, and using polar coordinates, we obtain
\begin{align*}
I(v) & = 2^{-d/2} \int_{\R^{2d}}
    \delta\left(|y|^2-   \frac{|v-v_1|^2}{2} \right)
    f\left(v, v_1, \frac{v+v_1}{2} + \frac{y}{\sqrt{2}}, \, \frac{v+v_1}{2} -\frac{y}{\sqrt{2}} \right) dv_1 dy\\
& = 2^{-d/2} \int_{\R^d} \int_0^\infty \int_{\S^{d-1}}
    \delta\left(r^2 - \frac{|v-v_1|^2}{2} \right)  f\left(v, v_1, \frac{v+v_1}{2} + \frac{r\sigma}{\sqrt{2}}, \, \frac{v+v_1}{2} -\frac{r\sigma}{\sqrt{2}} \right)
    r^{d-1} d\sigma dr dv_1.
\end{align*}
Next we apply another change of variables $z=r^2$ (and so $dr = \frac{dz}{2\sqrt{z}}$)
to further obtain
\begin{align*}
 I(v) 
& =   
2^{-\frac{d}{2} -1} \int_{\R^d} \int_0^\infty \int_{\S^{d-1}}
    \delta\left(z - \frac{|v-v_1|^2}{2} \right)  f\left(v, v_1, \frac{v+v_1}{2} + \frac{\sqrt{z}\sigma}{\sqrt{2}}, \, \frac{v+v_1}{2} -\frac{\sqrt{z}\sigma}{\sqrt{2}} \right)
    z^{\frac{d-2}{2}} d\sigma dz dv_1\\
& = 2^{-\frac{d}{2} -1}\, 2^{-\frac{d-2}{2}} \int_{\R^d}  \int_{\S^{d-1}}
     |v-v_1|^{d-2}~f\left(v, v_1, \frac{v+v_1}{2} +  \frac{|v-v_1|}{2}\sigma, \, \frac{v+v_1}{2} - \frac{|v-v_1|}{2}\sigma \right)
     d\sigma  dv_1,
\end{align*}
which completes the proof of the lemma.
\end{proof}

The following two lemmata, which are a consequence of Lemma \ref{convolution lemma} and \ref{delta lemma}, will be essential in providing estimates on velocity integrals in Section \ref{section - a priori estimates}. 

\begin{lemma}[Analogue of Lemma \ref{convolution lemma} in delta notation]\label{delta convolution lemma}
Let $d\in\{2,3\}$ and $q>2d-3$. Then there exists a positive constant $\tild{L}_q$ such that for $\Sigma = v+v_1 -v_2-v_3$, $\Omega = |v|^2 + |v_1|^2 - |v_2|^2 - |v_3|^2$ we have
\begin{align} 
    \int_{\R^{3d}} \delta(\Sigma)\delta(\Omega) 
    \frac{1}{ |v-v_1| \l  v_{1}\r^q }
    \,dv_{1}\,dv_{2}\,dv_{3} \le \tild{L}_q, \quad \forall v \in \mathbb{R}^d.
\end{align}
One can take 
\begin{align}\label{constant L}
\tild{L}_q = 2^{-d} \omega_{d-1}^2 \left( \frac{1}{d} + \frac{1}{2d-3} + \frac{2}{q-2d+3}\right),
\end{align} where $\omega_{d-1}$ denotes the area of the unit sphere $\S^{d-1}$ in $\R^d$. In particular, for $d=3$ we have
$\tild{L}_q  = 4\pi^2 \left( \frac{1}{3} + \frac{1}{q-3}\right)$.

\end{lemma}
\begin{proof}
By  Lemma \ref{delta lemma}, we have 
\begin{align*}
 I := \int_{\R^{3d}} &\delta(\Sigma)  \delta(\Omega) 
    \frac{1}{ |v-v_1| \l  v_{1}\r^q }
    \,dv_{1}\,dv_{2}\,dv_{3}    = 2^{-d} \int_{\R^d\times\S^{d-1}}|v-v_1|^{d-3} \l v_1\r^{-q} \,d\sigma\,dv_1.
\end{align*}
Since $d\in\{2,3\}$ and $q>2d-3$, we can apply Lemma \ref{convolution lemma} with $\delta = d-3$ to obtain
\begin{align*}
    I \le 2^{-d} \omega_{d-1} L_{q,d-3},
\end{align*}
where  $L_{q,d-3}$ is given by \eqref{L_{q,gamma}} with $\delta = d-3$. This completes the proof of the lemma.
\end{proof}

\begin{lemma}\label{appendix lemma on velocities weight}
For $d=3$ and $q>3$ there exists a positive constant $U_{q}$ such that

\begin{align}
&\sup_{v\in\R^d}\int_{\R^{3d}}\frac{\delta(\Sigma)\delta(\Omega)}{|v-v_1|\sqrt{1-(\frac{v-v_{1}}{|v-v_{1}|}\cdot\frac{v_{2}-v_{3}}{|v_{2}-v_{3}|})^2}}\frac{\l v\r^{q}}{\l v_1\r^q\l v_{2}\r^{q}\l v_{3}\r^{q}}\,dv_1\,dv_2\,dv_3 \leq U_{q}, \label{Uq three v}\\
&
\sup_{v\in\R^d}\int_{\R^{3d}}\frac{\delta(\Sigma)\delta(\Omega)}{|v-v_1|\sqrt{1-(\frac{v-v_{1}}{|v-v_{1}|}\cdot\frac{v_{2}-v_{3}}{|v_{2}-v_{3}|})^2}}\frac{1}{\l v_2\r^q \l v_3\r^q}\,dv_1\,dv_2\,dv_3 \leq U_{q}.  \label{Uq two v}
\end{align}
 One can take 
\begin{align} \label{constant U}
U_q =  2\pi^3 \left( \frac{1}{3} + \frac{1}{q-3}\right).
\end{align}

\end{lemma}
\begin{proof}
Let us denote the integral appearing in \eqref{Uq three v} by $I_1(v)$ and the integral in \eqref{Uq two v} by $I_2(v)$.
 By the conservation of energy, we have
$$\l v_2\r^2 \l v_3\r^2 
=(1+|v_2|^2)(1+|v_3|^2)
=1+|v|^2+|v_1|^2+ |v_2|^2|v_3|^2\geq \max\{\l v\r^2, \l v_1\r^2\},$$
and thus
$$\frac{\l v\r^q}{\l v_2\r^q \l v_3\r^q}\leq 1, \quad \text{and} \quad \frac{1}{\l v_2\r^2 \l v_3\r^2} \le \frac{1}{\l v_1\r^2},$$
which implies that both $I_1(v)$ and $I_2(v)$ can be estimated by the same upper bound:
\begin{align*}
I_1(v), \, I_2(v) 
&\le 
\int_{\R^{3d}}\frac{\delta(\Sigma)\delta(\Omega)}{|v-v_1|\sqrt{1-(\frac{v-v_{1}}{|v-v_{1}|}\cdot\frac{v_{2}-v_{3}}{|v_{2}-v_{3}|})^2}}\frac{1}{\l v_1\r^q }\,dv_1\,dv_2\,dv_3.
\end{align*}
Then by Lemma \ref{delta lemma}, we have
\begin{align*}
I_1(v), \, I_2(v) & \le 2^{-d} \int_{\R^d} \frac{|v-v_1|^{d-3}}{\l v_1\r^q }
\int_{\S^{d-1}}
\frac{1}{\sqrt{1-(\frac{v-v_{1}}{|v-v_{1}|}\cdot\sigma)^2}} d\sigma dv_1.
\end{align*}
For $d\ge 3$, integration in spherical coordinates yields
\begin{align}\label{b integral}
\int_{\S^{d-1}} \frac{1}{\sqrt{1-(\hat{n}\cdot\sigma)^2}}\,d\sigma
&= \omega_{d-2}\int_0^\pi \sin^{d-3}(\theta)\,d\theta \leq  \pi\omega_{d-2},  
\end{align}
where $\omega_{d-2}$ denotes the area of the unit sphere $\S^{d-2}$.
We next apply Lemma \ref{convolution lemma} with $\delta = d-3$, which requires that $d\in\{2,3\}$ and $q>2d-3$. Since the estimate \eqref{b integral} required $d\ge 3$,  we need the dimension to be $d=3$, and together with $q>2d-3=3$, we have
\begin{align*}
I_1(v), \, I_2(v) & \le 2^{-d} \pi \omega_{d-2} \int_{\R^d} |v-v_1|^{d-3}\l v_1\r^{-q} dv_1
\le 2^{-d} \pi \omega_{d-2} L_{q,d-3},
\end{align*}
where  $L_{q,d-3}$ is the constant  from Lemma \ref{convolution lemma}. Since $d=3$, we in fact have
\begin{align*}
    I_1(v), \, I_2(v) \le 2^{-3} \pi (2\pi) (4\pi) \left( \frac{1}{3} +\frac{1}{3} +\frac{2}{q-3} \right) = 2\pi^3 \left( \frac{1}{3} + \frac{1}{q-3}\right).
\end{align*}

This completes the proof of the lemma.
\end{proof}


\begin{thebibliography}{99}


 \bibitem{al09} R. Alonso, {\em Existence of global solutions to the Cauchy problem for the inelastic Boltzmann equation with near-vacuum data}, Indiana University Mathematics Journal Vol. 58, No. 3 (2009), pp. 999-1022.


 \bibitem{alga09} R. Alonso, I. M. Gamba, {\em Distributional and Classical Solutions to the Cauchy Boltzmann Problem for Soft Potentials with Integrable Angular Cross Section.} 
 J. Stat. Phys. (2009) 137: 1147–1165.

\bibitem{amliro20}
Z. Ammari, Q. Liard, C. Rouffort, 
{\em On well-posedness for general hierarchy equations of Gross-Pitaevskii and Hartree type.} Arch. Ration. Mech. Anal. 238 (2020), no. 2, 845–900.


\bibitem{am24} I. Ampatzoglou, {\em Global well-posedness and stability of the inhomogeneous kinetic wave equation near vacuum}, Kinet. Relat. Models (2024) DOI: 10.3934/krm.2024003

\bibitem{amcoge24}
I. Ampatzoglou, C. Collot, P. Germain, {\em Derivation of the kinetic wave equation for quadratic dispersive problems in the inhomogeneous setting.} to appear in Amer. J. Math. (2024)

\bibitem{ammipata24}
I. Ampatzoglou, J. K. Miller, N. Pavlovi\'c, M. Taskovi\'c,
{\em On the global in time existence and uniqueness of solutions to the Boltzmann hierarchy.}
arXiv:2402.00765


\bibitem{beto85} N. Bellomo, G. Toscani, 
{\em On the Cauchy problem for the nonlinear Boltzmann equation: global existence, uniqueness and asymptotic stability.}
J. Math. Phys. 26 (1985), no. 2, 334–338. 

\bibitem{bugehash21}
T. Buckmaster, P.  Germain, Z. Hani, J. Shatah, {\em Onset of the wave turbulence description of the longtime behavior of the nonlinear Schr\"odinger equation.} Invent. Math. 225 (2021), no. 3, 787–855.


\bibitem{chho23}
X. Chen, J. Holmer
{\em The Derivation of the Boltzmann Equation from Quantum Many-body Dynamics.}
arXiv:2312.08239


\bibitem{chpa11}
T. Chen, N. Pavlovi\'c,
{\em The quintic NLS as the mean field limit of a boson gas with three-body interactions.}  J. Funct. Anal. 260 (2011), no. 4, 959–997.



\bibitem{chpa14} T. Chen, N. Pavlovi\'c, 
{\em Derivation of the cubic NLS and Gross-Pitaevskii hierarchy from manybody dynamics in d=3 based on spacetime norms.}
Ann. Henri Poincar\'e 15 (2014), no.3, 543–588.

\bibitem{codige24} C. Collot, H. Dietert, P. Germain {\em Stability and Cascades for the Kolmogorov–Zakharov Spectrum of Wave Turbulence}, Arch. Rational Mech. Anal. Vol. 248, no. 7 (2024) 

\bibitem{coge19}
C. Collot, P. Germain, 
{\em On the derivation of the homogeneous kinetic wave equation.} arXiv:1912.10368

\bibitem{coge20}
C. Collot, P. Germain,
{\em Derivation of the homogeneous kinetic wave equation: longer time scales.}
arXiv:2007.03508

\bibitem{deha21}
Y. Deng, Z. Hani,
{\em On the derivation of the wave kinetic equation for NLS.} Forum Math. Pi 9 (2021), Paper No. e6, 37 pp.

\bibitem{deha23}
 Y. Deng, Z. Hani, {\em Full derivation of the wave kinetic equation.} Invent. Math. 233 (2023), no. 2, 543–724.

\bibitem{deha22} Y. Deng, Z. Hani,
{\em Propagation of chaos and the higher order statistics in the wave kinetic theory.}
arXiv:2110.04565



\bibitem{deha23long}
Y. Deng, Z. Hani,
{\em Long time justification of wave turbulence theory.}
arXiv:2311.10082



\bibitem{erscya06} L. Erd\"os, B. Schlein, H.-T. Yau
{\em Derivation of the Gross-Pitaevskii hierarchy for the dynamics of
   Bose-Einstein condensate.}
   Comm. Pure Appl. Math. 59 (2006), no. 12, 1659–1741.

\bibitem{erscya07} L. Erd\"os, B. Schlein, H.-T. Yau
{\em Derivation of the cubic non-linear Schrödinger equation from quantum dynamics of many-body systems. }
 Invent. Math. 167 (2007), no. 3, 515–614.


 \bibitem{esve15}
 M. Escobedo, J. J. L. Velázquez,  {\em Finite time blow-up and condensation for the bosonic Nordheim equation.} Invent. Math. 200 (2015), no. 3, 761–847.



\bibitem{fa20}
E. Faou, 
{\em Linearized wave turbulence convergence results for three-wave systems.} Comm.  Math. Phys. 378 (2): 807–849, 2020.
 

\bibitem{foleso18} S. Fournais, M. Lewin, J.P. Solovej,
{\em The Semi-classical Limit of Large Fermionic Systems.}
Cal. of Var. \textbf{57} 105 (2018).


\bibitem{gastte13} I. Gallagher, L. Saint-Raymond, B. Texier,
{\em From Newton to Boltzmann: Hard Spheres and Short-range Potentials.} Zurich Lectures in Advanced Mathematics. European Mathematical Society, 2013.

\bibitem{geiotr20} P. Germain, A. D. Ionescu, M.-B. Tran, 
{\em Optimal local well-posedness theory for the kinetic wave equation.} 
J. Funct. Anal. 279 (2020), no. 4, 108570, 28 pp.


\bibitem{harosttr22} A. Hannani, M. Rosenzweig, G. Staffilani, M.-B. Tran
{\em On the wave turbulence theory for a stochastic KdV type equation -- Generalization for the inhomogeneous kinetic limit.} 
arXiv:2210.17445


\bibitem{ha62} K. Hasselmann, {\em On the non-linear energy transfer in a gravity-wave spectrum. I. General theory.} J. Fluid Mech. 12 (1962), 481–500.

\bibitem{ha63} K. Hasselmann, {\em On the non-linear energy transfer in a gravity wave spectrum. II. Conservation theorems; wave-particle analogy; irreversibility.} J. Fluid Mech. 15 (1963), 273–281. 

\bibitem{hashzh23} Z. Hani, J. Shatah, H. Zhu, {\em Inhomogeneous turbulence for the Wick nonlinear Schrödinger equation}, 	To appear in Comm. Pure Appl. Math.   https://doi.org/10.1002/cpa.22198


\bibitem{hs55}E. Hewitt, L. Savage, {\em Symmetric measures on Cartesian products}. Trans. Amer. Math. Soc. 80
(1955), p. 470-501.


\bibitem{ilsh84} R. Illner, M. Shinbrot, 
{\em The Boltzmann Equation: Global Existence for a Rare Gas in an Infinite Vacuum.}
Commun. Math. Phys. 95, 217-226 (1984).


\bibitem{kash78} S. Kaniel, M. Shinbrot, 
{\em The Boltzmann Equation: Uniqueness and Local Existence.}
Comm. Math. Phys. 58, 65-84 (1978).

\bibitem{ki75} F. King, 
{\em BBGKY hierarchy for positive potentials.}
Ph.D. dissertation, Dept. Mathematics, Univ. California, Berkeley, 1975.

\bibitem{kiscst11}
K. Kirkpatrick, B. Schlein, G. Staffilani, 
{\em Derivation of the two-dimensional nonlinear Schrödinger equation from many body quantum dynamics.}
Amer. J. Math. 133 (2011), no. 1, 91–130. 

\bibitem{klma08}
S. Klainerman, M. Machedon,
{\em On the uniqueness of solutions to the Gross-Pitaevskii hierarchy.}
Comm. Math. Phys. 279 (2008), 169-185.



\bibitem{la75} O. E. Lanford III,
{\em Time evolution of large classical systems.}
Dynamical systems, theory and applications (Rencontres, Battelle Res. Inst., Seattle, Wash., 1974), pp. 1–111, Lecture Notes in Phys., Vol. 38, Springer, Berlin-New York, 1975. 

\bibitem{lusp11}
J. Lukkarinen, H. Spohn, {\em Weakly nonlinear Schrödinger equation with random initial data.} Invent. Math. 183 (2011), no. 1, 79–188.
\bibitem{me23} A. Menegaki, 
{\em $L^2$-stability near equilibrium for the 4 waves kinetic equation}, Kinet. Relat. Models, (2023). doi: 10.3934/krm.2023031.
\bibitem{na11} S. Nazarenko, {\em Wave turbulence.} Lecture Notes in Physics, 825. Springer, Heidelberg, 2011.

\bibitem{neru13} A. C. Newell, B. Rumpf, {\em Wave turbulence: a story far from over.} Advances in wave turbulence Volume 83 of World Sci. Ser. Nonlinear Sci. Ser. A Monogr. Treatises, 1-51. World Sci. Publ., Hackensack, NJ, 2013.

\bibitem{pe29}
R. Peierls.{\em Zur kinetischen theorie der Wärmeleitung in kristallen.}
Annalen der Physik 395 (1929) 1055-1101.

\bibitem{pusasi14} M. Pulvirenti, C. Saffirio, S. Simonella, 
{\em On the validity of the Boltzmann equation for short range potentials.}
Rev. Math. Phys. 26 (2014), no 2, 1450001.

\bibitem{rost22}
M. Rosenzweig, G. Staffilani,  {\em Uniqueness of solutions to the spectral hierarchy in kinetic wave turbulence theory.}
Phys. D 433 (2022), Paper No. 133-148, 16 pp.

\bibitem{smja13}
P. B. Smit, T. T. Janssen,
{\em The Evolution of Inhomogeneous Wave Statistics through a Variable Medium.} Journal of Physical Oceanography, Vol 43, Issue 8, pp 1741–1758.

\bibitem{so15} V. Sohinger,
{\em A rigorous derivation of the defocusing cubic nonlinear Schrödinger equation on T3 from the dynamics of many-body quantum systems.}
Ann. Inst. H. Poincaré C Anal. Non Lin\'eaire 32 (2015), no. 6, 1337–1365. 


\bibitem{sp06} H. Spohn
{\em The phonon Boltzmann equation, properties and link to weakly anharmonic lattice dynamics.}
J. Stat. Phys. 124 (2006), no. 2-4, 1041–1104.

\bibitem{sttr21} G. Staffilani, M.-B. Tran
{\em On the wave turbulence theory for a stochastic KdV type equation.}
arXiv:2106.09819


\bibitem{to86}  G. Toscani, 
{\em On the nonlinear Boltzmann equation in unbounded domains.}
Arch. Rational Mech. Anal. 95 (1986), no. 1, 37–49.


 \bibitem{to88} G. Toscani, {\em Global solution of the initial value problem for the Boltzmann equation near a local Maxwellian}, Arch. Rational Mech. Anal. 102 (1988), no. 3, 231–241.

 \bibitem{zalv75}
 V. Zakharov, V. L’vov , The statistical description of nonlinear wave fields, RaF, 18 (1975), 1470-1487.

\bibitem{zalvfa92}
V. E. Zakharov, V. S. L’vov, G. Falkovich, {\em Kolmogorov spectra of turbulence I: Wave turbulence.} Springer
Science \& Business Media, 1992.
\end{thebibliography}
\end{document}